\documentclass[a4paper,12pt]{article}

\usepackage[austrian,american]{babel}

\usepackage[intlimits,leqno]{amsmath}
\usepackage{amssymb} 
\usepackage{amsthm}
\usepackage{graphicx,color}
\usepackage{makeidx}
\usepackage{epsfig}
\usepackage{latexsym}
\usepackage{pstricks,pst-node,amsfonts}

\theoremstyle{plain}
\newtheorem{theo}{Theorem}
\newtheorem{prop}{Proposition}

\newtheorem{lemm}{Lemma}
\newtheorem{rema}{Remark}
\newtheorem{exam}{Example}

\newcommand\R{{\mathbb{R}}}
\newcommand\C{{\mathbb{C}}}
\newcommand{\F}{\mathcal{F}}
\newcommand{\J}{\mathcal{J}} 
\renewcommand{\i}{{\rm i}}
\newcommand{\x}{\mathbf{x}}
\renewcommand{\k}{\mathbf{k}}
\newcommand{\y}{\mathbf{y}}
\renewcommand{\d}{{\rm d}}
\newcommand{\supp}{\mbox{supp}}

\begin{document}

\title{On time reversal in photoacoustic tomography for tissue similar to water}

\author{ Richard Kowar\\
Department of Mathematics, University of Innsbruck, \\
Technikerstrasse 21a, A-6020,Innsbruck, Austria
}

\maketitle

\begin{abstract}
This paper is concerned with time reversal in \emph{photoacoustic tomography} (PAT) of dissipative media that are similar 
to water. Under an appropriate condition, it is shown that the time reversal method in~\cite{Wa11,AmBrGaWa11} based on the 
non-causal thermo-viscous wave equation can be used if the non-causal data is replaced by a \emph{time shifted} set of 
causal data. 
We investigate a similar imaging functional for time reversal and an operator equation with the time reversal image as right 
hand side. If required, an enhanced image can be obtained by solving this operator equation. 
Although time reversal (for noise-free data) does not lead to the exact initial pressure function, the theoretical and numerical 
results of this paper show that regularized time reversal in dissipative media similar to water is a valuable method. 
We note that the presented time reversal method can be considered as an alternative to the causal approach in~\cite{KaSc13} and 
a similar operator equation may hold for their approach.  
\end{abstract}

\section{Introduction}\label{sec-intro}

\emph{Photoacoustic tomography}~\cite{BurMatHalPal07,FinPatRak04,HaScBuPa,KucKun08,
ScGrLeGrHa09,XuWan05,XuWan05b,XuWanAmbKuc03} is a very promising medical imaging method that is currently improved by taking into 
account more complicated tissue properties . In this paper we focus on a time reversal method in dissipative 
media~\cite{AmBrGaWa11,AmBrGaWa13,BurGruHalNusPal07,HriKucNgu08,KaSc13,Ko10,KoSc12,RivZhaAna06,TrZhCo10,Wa11} and 
show that under an appropriate condition  
\begin{itemize}
\item the imaging functional in~\cite{AmBrGaWa11} based on the non-causal \emph{thermo-viscous wave equation} can be 
      used for time reversal if the non-causal data is replaced by a \emph{time shifted} set of causal pressure data and that 

\item a similar imaging functional can be used for time reversal, which can be improved by solving an operator equation 
      with the time reversal result as right hand side.
\end{itemize}        
We note that the \emph{time shift relation} between causal and non-causal pressure data holds only approximately, but  
in principle the non-causal data can be calculated from the causal one. However the use of the \emph{time shift relation}  
is less computational expensive and, in addition, it explains the successful use of the non-causal thermo-viscous wave equation. 
Consider a fixed experimental set-up for which many experiments are performed. It seems natural that the experimenter
performs - as a matter of routine - one and the same pressure data \emph{offsets} to all data. If this data offset 
corresponds to the mentioned time shift, then causality is approximately "restored" and causality violations due to 
the thermo-viscous wave equation may not be observed.

To outline the contents of this paper, we start with a basic description of the inverse problem of PAT in dissipative media.

\subsection{The PAT problem in dissipative media}
 
The goal of PAT is to estimate the function $\varphi:\R^3\to\R$ with 
$\supp(\varphi)\subseteq\Omega$ from pressure data $p$ measured at a boundary, say $\partial \Omega$, in which
$p$ and $\varphi$ are related by  (cf.~\cite{Ko10,KoSc12})
\begin{equation}\label{waveeqp}
\begin{aligned}
    \Delta p - \left(D_* + \frac{1}{c_0}\,\frac{\partial}{\partial t}\right)^2\,p 
        = - \frac{\varphi}{c_0^2}\delta'(t) \quad \mbox{ on }\quad \R^4 
        \quad\mbox{with} \quad
    p|_{t<0} = 0 \,,
\end{aligned}
\end{equation}
where $c_0$ denotes the speed of sound, $D_*$ is a time convolution operator with kernel $\F^{-1}\{\alpha_*\}$ and 
$\alpha_*$ is the complex attenuation law of the medium in which the wave propagates. We focus on the following complex 
attenuation laws  
\begin{equation}\label{Modeltvksb}
\begin{aligned} 
   \alpha_*^{ksb} (\omega) 
      := \frac{(-\i\,\omega)}{c_0\,\sqrt{1+(-\i\,\tau_1\,\omega)}} 
         + \alpha_2\,(-\i\,\omega) \qquad\quad (\alpha_2\geq 0)\,
\end{aligned}
\end{equation}
and 
\begin{equation}\label{Modeltv}
\begin{aligned} 
   \alpha_*^{tv} (\omega) 
      := \frac{(-\i\,\omega)}{c_0\,\sqrt{1+(-\i\,\tau_1\omega)}} 
          - \frac{(-\i\,\omega)}{c_0}  \qquad\quad (c_0,\tau_1>0)\,.
\end{aligned}
\end{equation}
Here $\tau_1$ denotes the relaxation time. Only the first one obeys causality, i.e. the respective pressure function $p^{ksb}$ 
has a finite wave front speed (cf.~\cite{Ko10b,KoSc12})
$$
        c_F = \frac{c_0}{1+\alpha_2\,c_0}  \leq c_0\,.
$$
As mentioned above, it will be sown that 
$$
     p^{tv}(\x,t) \approx p^{ksb}\left( \x, t + \frac{|\x|}{c_0} + \alpha_2\,|\x| \right)   
              \qquad \mbox{for}\qquad 
     \x\in\partial \Omega,\,t\in\R\,
$$
holds under an appropriate condition. Hence all results based on the non-causal thermo-viscous wave equation can be applied to 
time shifted causal data. For this setting we introduce an imaging functional for estimating the initial pressure function 
$\varphi$.

\subsection{The imaging functional}

Let $p^{tv}$ be the solution of~(\ref{waveeqp}) with~(\ref{Modeltv}) 
and $\varphi\in L^2(\Omega)$, 
$$
    \phi_T := p^{tv}|_{t=T}\,,\qquad 
    \beta := p^{tv}|_{\partial \Omega} \qquad\mbox{and}\qquad 
    \Phi_T:(\phi_T|_{\Omega},\beta)\mapsto p^{tv}|_{t=T}\,.
$$
Moreover, let $q^{tv}$ denote the solution of the \emph{regularized time reversed} thermo-viscous wave equation 
\begin{equation}\label{tvwaveeqrev}
     \left(\mbox{Id} - \tau_1\frac{\partial}{\partial t}\right)\,\Delta\,q 
     - \frac{1}{c_0^2}\,\frac{\partial^2 q}{\partial t^2} 
     = \frac{\mathcal{R}_D\,(\phi_T)}{c_0^2}\,\left(\mbox{Id} - \tau_2\,\frac{\partial }{\partial t}\right)\,
            \delta'(t) \,,
\end{equation} 
where $\mathcal{R}_D$ is a regularization operator\footnote{A good practical choice is a convolution with a Gaussian function 
with expected value $0$ and variance $2\,D$ with the necessary side condition $D>\tau_1\,c_0^2\,T$.} that guarantees the 
existence of the time reversal for $\varphi\in L^2(\Omega)$. With these notations, we introduce the functionals 
\begin{equation}\label{defimagf}
     F[\mathcal{R}_D\,\phi_T] := q^{tv}|_{t=T} 
             \qquad\mbox{and}\qquad 
     F_1 [\phi_T|_{\Omega},\beta] := 2\,F[\mathcal{R}_D\,\Phi_T(\phi_T|_{\Omega},\beta)]\,.
\end{equation} 
In this paper we show under the assumption
\begin{equation}\label{assT}
       \supp(\varphi) \subseteq B_{2\,c_0\,T}(\x)  \qquad\mbox{for all}\qquad \x\in\Omega
\end{equation}
that the \emph{imaging functional} $F_1$ satisfies for $D>0$
\begin{equation*}
   \lim_{\tau_1\to 0} F_1[\phi_T|_{\Omega},\beta] 
           = \mathcal{R}_D\varphi 
           = F_1[\phi_T|_{\Omega},\beta]|_{\tau_1=0}    \qquad \mbox{on}\qquad \Omega\,, 
\end{equation*}
in particular 
\begin{equation}\label{approxF}
     F_1[\phi_T|_{\Omega},\beta] \sim  \mathcal{R}_D \,\varphi  \,  
     \qquad\quad \mbox{for sufficiently small $\tau_1>0$} \,, 
\end{equation}
and that $\mathcal{R}_D \,\varphi$ satisfies the operator equation  
\begin{equation}\label{fredinteq}
\begin{aligned}
   \left(\mbox{Id} + \tau_1^2\,c_0^2\,\Delta\right)^2\,\mathcal{R}_D\,\varphi 
      + c_0^2\,\Delta\,\J_T\,\mathcal{R}_D\,\varphi 
            = \frac{1}{2}\,F_1[\phi_T|_{\Omega},\beta]\qquad \mbox{on $\in\Omega$}\,,
\end{aligned}
\end{equation}
where $\J_t\,\xi$ is defined as the solution of 
\begin{equation}\label{waveequ}
     \Delta\,w 
        + \frac{\tau_1^2\,c_0^2}{4}\,\Delta^4\,w 
        - \frac{1}{c_1^2}\,\frac{\partial^2 w}{\partial t^2} 
          = - \frac{2\,\xi}{c_1^2}  \qquad\mbox{on $\R^3\times [0,T]$} \qquad (c_1:=2\,c_0)\,.
\end{equation} 
Subsequently, the properties of operator equation~(\ref{fredinteq}) and wave equation~(\ref{waveequ}) are investigated.  

We interpret the appearance of the regularization $\mathcal{R}_D$ with restriction $D>\tau_1\,c_0^2\,T$ as a resolution limit 
for the noise-free case. Stronger attenuation or a larger domain (larger $T$) only permit the estimation of a stronger smoothed 
version of the initial pressure function $\varphi$. For tissue similar to water we have $c_0\approx 1500\,\frac{m}{s}$, 
$\tau_1\approx 10^{-9}\,s$ (cf.~\cite{KiFrCoSa00}) and if the diameter of $\Omega$ is $d=0.5\,m$, i.e. $T=\frac{0.5\,m}{c_0}$, 
then we get 
$$
         D > 7.5\cdot10^{-7}\,m\,.  
$$
If $T$ is four times as large, then $D > 3\cdot10^{-6}\,m$, which is from the practical point of view not a limitation for PAT.

In~\cite{AmBrGaWa11} and~\cite{KaSc13} similar imaging functionals were considered for non-causal and causal data, 
respectively, and a result analogous to~(\ref{approxF}) was shown via approximation in the frequency-space domain. 
One goal of this paper is to get an expression for a reasonable imaging functional without cutting off frequencies or 
wave vectors, because the cut off values are not known in general. In addition, we want to show how an approach based on 
non-causal thermo-viscous data can be used if real causal data are available. Moreover, it seems less computational expensive 
to solve a time reversed partial differential equation than an time reversed integro-differential equation. However, 
in the long run, it seems indispensable to base some of the high quality imaging methods on causal integro-differential 
equations (cf.~\cite{KaSc13}). 
We consider the work presented in this paper as a trade off between the approaches~\cite{AmBrGaWa11} and~\cite{KaSc13}. \\

This paper is organized as follows: In Section~\ref{sec-dp} the properties of dissipative waves propagating forward and 
backward in time are investigated. Afterwards, it is shown that the operator $\J_T$ is well-defined and compact. 
The operator equation~(\ref{fredinteq}) is derived in Section~\ref{sec-dopeq} and its properties are investigated in 
the subsequent section. In Section~\ref{sec-propim} it is shown that the imaging functional converges to the 
initial pressure function for $\tau_1\to 0$. Numerical simulations of the imaging functional are presented in 
Section~\ref{sec-sim} and the paper is concluded with the section Conclusions.

\section{Properties of waves in thermo-viscous media}\label{sec-dp}

This section investigate waves in thermo-viscous media that propagates forward and backward in time.

\subsection{Waves forward in time}

Until Lemma~\ref{lemm:tva} of this section, $\hat p(\cdot,\omega)$ and $\F\{p\}(\cdot,\omega)$ denote the Fourier transform 
of $p(\cdot,t)$ with respect to $t$. \\

Dissipative pressure waves for PAT can be modeled as the solution of the integro-differential 
equation~(\ref{waveeqp}), where $\alpha_*:\R\to\C$ is such that 
\begin{itemize}
\item [(A1)]  $\Re(\alpha_*)$ is even, non-negative and increasing on $(0,\infty)$,   

\item [(A2)]  $\Im(\alpha_*)$  is odd, 

\item [(A3)]  a spherical wave has a finite wave front speed smaller or equal to $c_0$, i.e. (cf.~\cite{Ko10,KoSc12})
$$
          \F^{-1}\{e^{-\alpha_*(\cdot)}\}(t) = 0 \qquad\mbox{for}\qquad 
          t\leq 0\,.
$$
\end{itemize} 
The real part of $\alpha_*:\R\to \C$ is called the \emph{attenuation law} and $\alpha_*$ is called the 
\emph{complex attenuation law}. 
The first two conditions guarantee that the Green function is real valued and the second is a causality condition on 
the Green function.

First we focus on the Green function of wave equation~(\ref{waveeqp}) for general complex attenuation laws and for the two 
special cases~(\ref{Modeltvksb}) and~(\ref{Modeltv}).

\begin{theo}\label{theo:direct1}
Let $T>0$, $\alpha_*$ satisfy properties (A1)-(A3) and $\varphi \in L^2(\R^3)$. Then the direct problem~(\ref{waveeqp})  
has a unique solution and the Green function is given by 
\begin{equation}\label{defGreen}
   \hat G (\x,\omega) 
     =  \frac{1}{\sqrt{2\,\pi}}\,\frac{e^{\i\,k(\omega)\,|\x|}}{4\,\pi\,|\x|} 
     \qquad\mbox{with}\qquad 
     k(\omega) = \i\,\alpha_*(\omega) + \frac{\omega}{c}\,.
\end{equation}
Moreover, causality condition (A3) is satisfied for $\alpha_*^{ksb}$ defined as in~(\ref{Modeltvksb}), but not for 
$\alpha_*^{tv}$ defined as in~(\ref{Modeltv}). 
\end{theo}

\begin{proof}
The proof of these facts can be found in~\cite{Ko10} and~\cite{KoSc12}. 
\end{proof}

Because the identity 
$$
    \left(1+(-\i\,\tau_1\,\omega)\right)\,
    \left( \alpha_*^{tv} + \frac{(-\i\,\omega)}{c_0}\right)^2 
         = \frac{(-\i\,\omega)^2}{c_0^2}  \qquad\quad \omega\in\R
$$
in the frequency domain is equivalent to the operator identity
$$
  \left(\mbox{Id} + \tau_1\,\frac{\partial}{\partial t} \right)\, 
     \left(D_*^{tv} + \frac{1}{c_0}\,\frac{\partial}{\partial t}\right)^2 
        = \frac{1}{c_0^2}\,\frac{\partial^2}{\partial t^2}\,,
$$
we see that wave equation~(\ref{waveeqp}) with complex attenuation law~(\ref{Modeltv}) is equivalent to the 
\emph{thermo-viscous wave equation}: 
\begin{equation}\label{tvwaveeq}
\begin{aligned}
     \left(\mbox{Id} + \tau_1\frac{\partial}{\partial t}\right)\,\Delta\,p 
     - \frac{1}{c_0^2}\,\frac{\partial^2 p}{\partial t^2} 
     = - \frac{\varphi}{c_0^2}\,\left(\mbox{Id} + \tau_2\,\frac{\partial }{\partial t}\right)\,
            \delta(t) \,
\end{aligned}
\end{equation} 
with $\tau_2=\tau_1$. Because the operator $\left(\mbox{Id} + \tau_2\,\frac{\partial }{\partial t}\right)$ 
is usually neglected in the literature, we distinguish between $\tau_1$ and $\tau_2$ in the following.   
This allows us to consider the approximate case $\tau_2=0$, too.  

The following Lemma shows that the Green functions of~(\ref{waveeqp})  with $\alpha_*$ defined 
by~(\ref{Modeltvksb}) and~(\ref{Modeltv}), respectively, are related by a space dependent time shift. 

\begin{lemm}\label{lemm:shift}
Let $G^{ksb}$ and $G^{tv}$ denote the Green functions of~(\ref{waveeqp})  with $\alpha_*$ defined 
by~(\ref{Modeltvksb}) and~(\ref{Modeltv}), respectively and 
\begin{equation}\label{defT1}
\begin{aligned} 
     T_1(\x) := \frac{|\x|}{c_1}   \qquad\mbox{with}\qquad  
     \frac{1}{c_1} := \frac{1}{c_0} + \alpha_2\,.
\end{aligned}
\end{equation} 
Then 
\begin{equation}\label{relGthGthc}
\begin{aligned} 
     G^{tv}(\x,t - T_1(\x)) = G^{ksb}\left( \x,t\right)   
              \qquad \mbox{for}\qquad 
     \x\in\R^3,\,t\in\R\,.
\end{aligned}
\end{equation}
\end{lemm}

\begin{proof}
According to~(\ref{Modeltvksb}) and~(\ref{Modeltv}), we have
$$   
    \alpha_*^{ksb} (\omega) 
        =  \alpha_*^{tv} (\omega) 
            + \left( \frac{1}{c_0} + \alpha_2 \right)\,(-\i\,\omega)\,
$$
and thus inserting these laws in~(\ref{defGreen}) yields   
$$
    \hat G^{ksb}(\x,\omega) 
        = \sqrt{2\,\pi}\,\hat G^{tv}(\x,\omega) \, 
        \frac{e^{\i\,\omega\,\left(\frac{|\x|}{c_0} + \alpha_2\,|\x|\right)}}{\sqrt{2\,\pi}}\,.
$$
The claim follows by employing the convolution theorem 
$$
     \F\{f\,g\} = \sqrt{2\,\pi}\,\hat f\,\hat g
                    \qquad\mbox{and}\qquad
     \F^{-1}\left\{\frac{e^{\i\,\omega\,\frac{|\x|}{c_1}}}{\sqrt{2\,\pi}}\right\}(t) 
          = \delta\left(t-\frac{|\x|}{c_1}\right)\,.
$$
\end{proof}

We note that $G^{tv}$ satisfies a partial differential equation, but $G^{ksp}$ does not. As a consequence, 
$p^{ksb}:=G^{ksp} *_{\x,t} f$ satisfies an integro-differential equation with forcing term $f$.

Now we are ready to prove that if the size of $\supp(\varphi)$ is much smaller than its distance 
to the detectors on $\partial \Omega$, then the non-causal and causal pressures $p^{tv}$ and $p^{ksb}$ are related by 
\begin{equation}\label{relp}
\begin{aligned} 
     p^{tv}(\x,t) \approx p^{ksb}\left( \x, t + T_1(\x) \right)   
              \qquad \mbox{for}\qquad 
     \x\in\partial \Omega,\,t\in\R\,
\end{aligned}
\end{equation}
with $T_1$ defined as in~(\ref{defT1}). This enables us to avoid the integro-differential equation for $p^{ksb}$. 

\begin{prop}\label{prop:shift2}
Let $p^{ksb}$ and $p^{tv}$ denote the solutions of~(\ref{waveeqp})  with $\alpha_*$ defined 
by~(\ref{Modeltvksb}) and~(\ref{Modeltv}), respectively.  
If 
\begin{equation}\label{mainass}
\begin{aligned} 
       T_1(\x) \approx T_1(\x-\y) \qquad\mbox{for all}\qquad \x\in\partial\Omega,\,\y\in \supp(\varphi)\,,
\end{aligned}
\end{equation}
then ~(\ref{relp}) holds. 
For example, condition~(\ref{mainass}) holds, if 
$$
      T_1(\y) <<  T_1(\x)  \qquad\mbox{for all}\qquad \x\in\partial\Omega,\,\y\in \supp(\varphi)\,. 
$$
\end{prop}

\begin{proof}
Let condition~(\ref{mainass}) hold and $d(\x,t):=\delta\left(t+T_1(\x)\right)$. Then we infer from Lemma~\ref{lemm:shift} 
that 
\begin{equation*}
\begin{aligned}
   c_0^2\,p^{ksb} *_t d  
      &= \int \frac{\partial G^{ksb}}{\partial t}(\cdot-\y,\cdot)\,\varphi(\y)\,\d \y  *_t d 
      = \int \frac{\partial G^{ksb}}{\partial t}\left(\cdot-\y,\cdot + T_1(\cdot)\right)\,\varphi(\y)\,\d \y \\
       &\;\approx\; \int \frac{\partial G^{tv}}{\partial t}\left(\cdot-\y,\cdot + T_1\left(\cdot - \y\right)\right)\,\varphi(\y)\,\d \y 
      = G^{tv} *_\x \varphi 
      = c_0^2\,p^{tv}  \,,
\end{aligned}
\end{equation*}
which proves the first claim. Now for the second one. Let $\x\in\partial\Omega$ and $\y\in \supp(\varphi)$. 
From the assumption, 
$$
     T_1(\x) := \frac{|\x|}{c_1}   \qquad\mbox{with}\qquad  
     \frac{1}{c_1} := \frac{1}{c_0} + \alpha_2\,
$$
and the two \emph{triangle inequalities}, it follows that
$$
    T_1(\x-\y) \leq T_1(\x) + T_1(\y) \approx T_1(\x)
$$
and 
$$
   T_1(\x - \y) \geq  |T_1(\x) - T_1(\y) | \approx T_1(\x) \,, 
$$
i.e. $T_1(\x-\y) \approx T_1(\x)$.
\end{proof}

\begin{rema}\label{rema:nsw}
Under appropriate conditions Proposition~\ref{prop:shift2} with 
$$
c_1 := c_0\,\left(\sqrt{\frac{\tau_1}{\tilde\tau_1}} - 1\right)^{-1}
$$  
holds for the causal wave equation of Nachman, Smith and Waag in~\cite{NaSmWa90}, 
if $\tilde\tau_1\approx \frac{1}{2}(3-\sqrt{5})\,\tau_1$.  
This follows from the fact that the complex attenuation law $\alpha_*^{nsw}(\omega)$ (cf.~\cite{KoSc12}) of the Nachman, 
Smith and Waag model is approximately the same as $\alpha_*^{tv}(\omega)+ \frac{(-\i\,\omega)}{c_1}$ for small frequencies. 
As a demonstration of the quality of this relation, we have plotted the real and imaginary parts of both laws in Fig.~\ref{fig:comp} 
for the parameter values $\tau_1= 10^{-9}\,s$, $\tilde\tau_1:=\frac{1}{2}(3-\sqrt{5})\,\tau_1$ and 
$c_0=\frac{3}{2}\cdot 10^3\,\frac{m}{s}$ (cf.~\cite{KiFrCoSa00,Web00}). 
\begin{figure}[!ht]
\begin{center}
\includegraphics[height=5.1cm,angle=0]{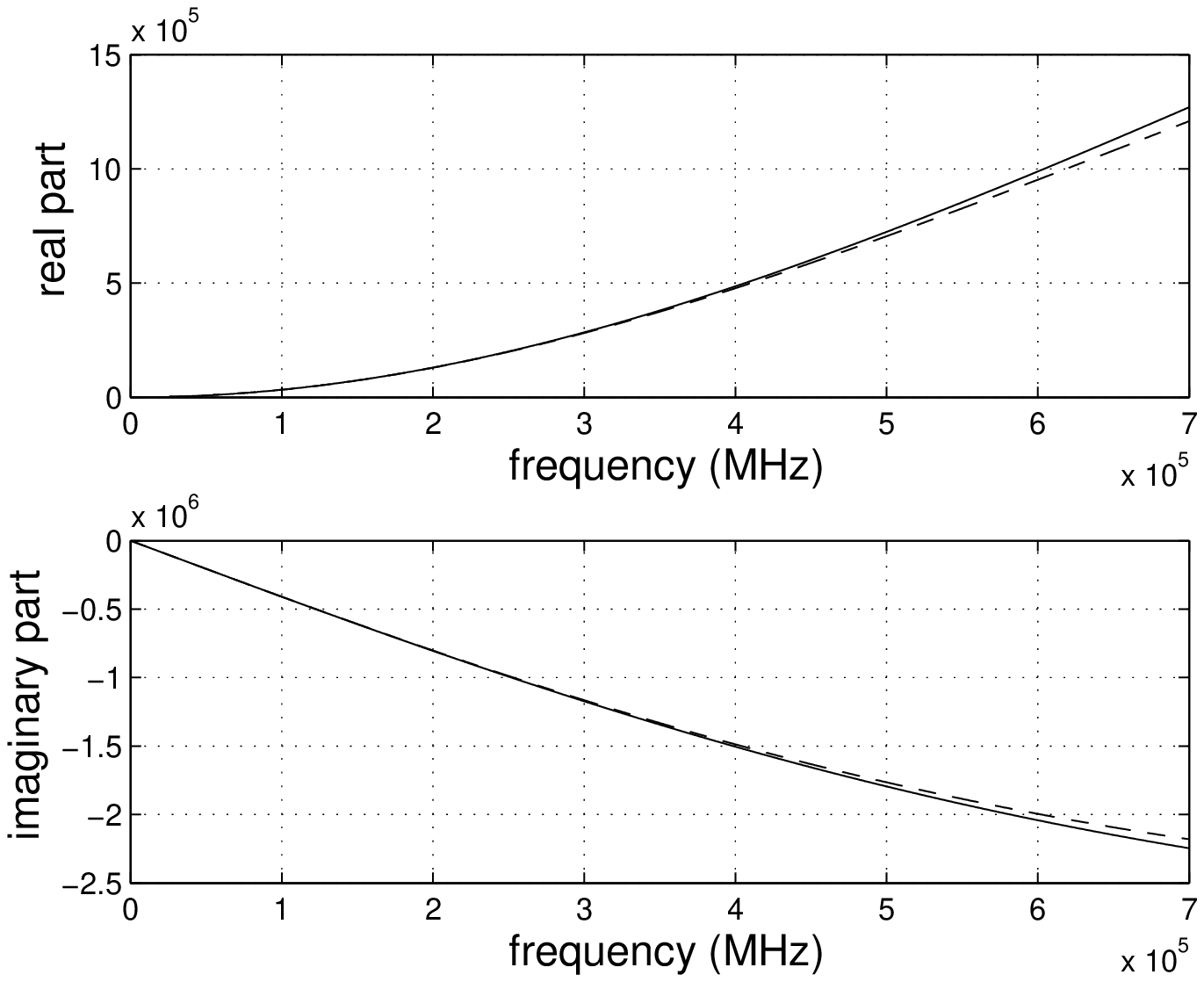}
\includegraphics[height=5.1cm,angle=0]{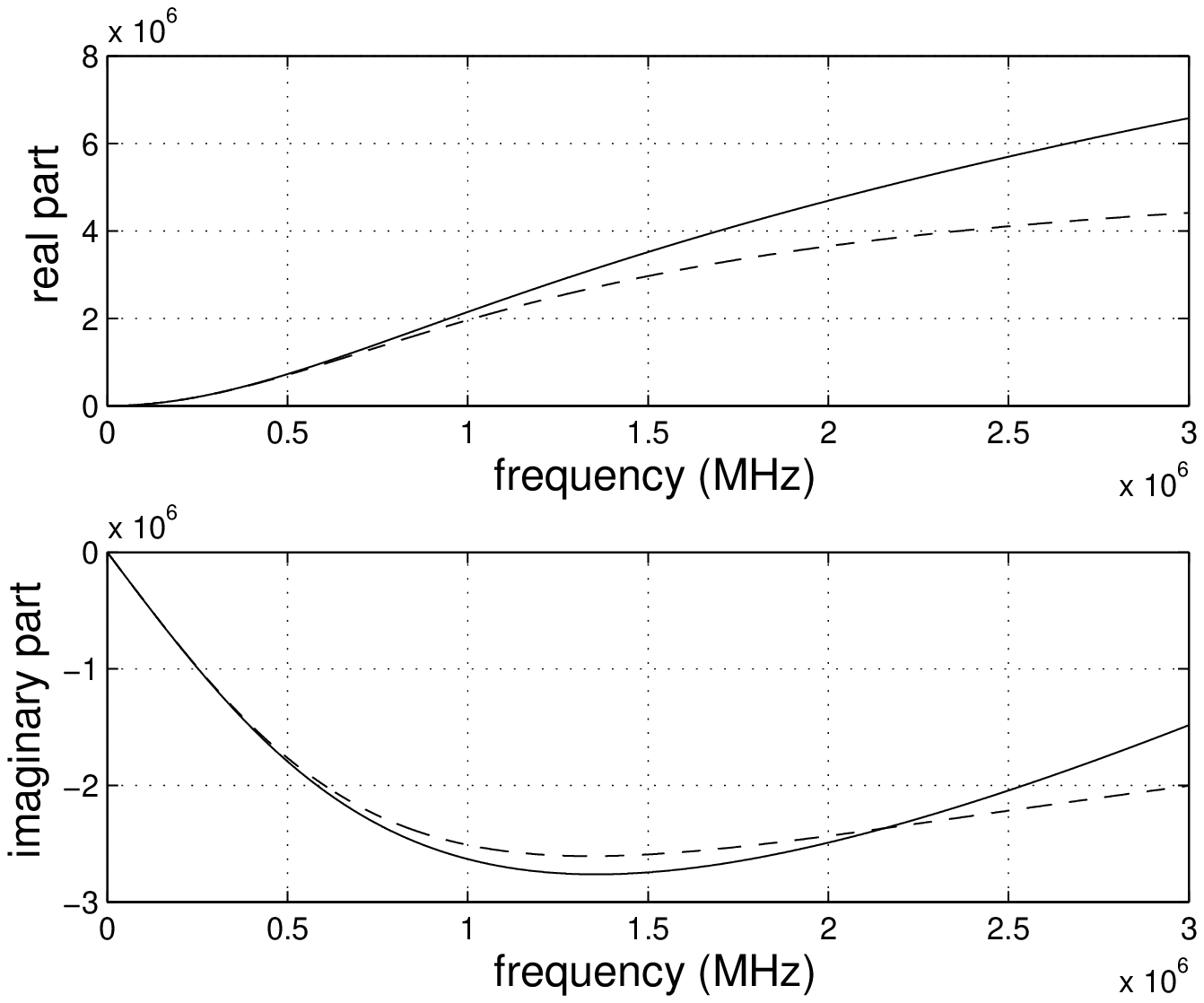}
\end{center}
\caption{Comparison of the real and imaginary parts of $\alpha_*^{nsw}(\omega)$ (dashed) and 
$\alpha_*^{tv}(\omega) + \frac{(-\i\,\omega)}{c_1}$ (full) for the frequency ranges 
$[0\,,\,7\cdot 10^5] MHz$ and $[0\,,\,3\cdot 10^6]$ MHz. Acoustic frequencies can be measured up to about $60\,MHz$.}
\label{fig:comp}
\end{figure}
\end{rema}

For the rest of this paper, $\hat p(\k,\cdot)$ and $\F\{p\}(\k,\cdot)$ denote the Fourier transform of $p(\x,\cdot)$ with 
respect to $\x$. 
To derive the operator equation~(\ref{fredinteq}) we require a representation of the solution of the thermo-viscous wave 
equation and its time reverse in the \emph{wave vector - time domain}.

\begin{lemm}\label{lemm:tva} 
Let 
\begin{equation}\label{defs}
\begin{aligned}
    \mbox{$c_0,\,\tau_1>0$, $\quad \tau_2\geq 0$, $\quad k_c:= \frac{2}{\tau_1\,c_0}$, 
    $\quad\mu(k) := c_0\,\frac{k^2}{k_c}\quad$ and } \\    
   \vartheta(k) := c_0\,k\,\sqrt{1-\frac{k^2}{k_c^2}} 
   \qquad\mbox{for} \qquad k>0 \,.
\end{aligned}
\end{equation}
If $p$ solves  
\begin{equation*}
\begin{aligned}
     \left(\mbox{Id} + \tau_1\frac{\partial}{\partial t}\right)\,\Delta\,p
     - \frac{1}{c_0^2}\,\frac{\partial^2 p}{\partial t^2} 
     = - \frac{\phi}{c_0^2}\,\left(\mbox{Id} + \tau_2\,\frac{\partial }{\partial t}\right)\,
            \delta(t)
\end{aligned}
\end{equation*} 
and $\x\mapsto p(\x,t)$ is tempered for $t\in\R$, then 
$$
   \hat p(\k,t) 
       = \hat\phi(\k)\,\left(\mbox{Id} + \tau_2\,\frac{\partial }{\partial t}\right)\left( 
          \frac{e^{-\mu(k)\,t}}{(2\,\pi)^{3/2}}\,\frac{ \sin(\vartheta(k)\,t)}{\vartheta(k)}\,H(t) \right) 
$$ 
for $k:=|\k|$ with $\k\in\R^3$. Here $H(t)$ denotes the \emph{Heaviside} function.
\end{lemm}

\begin{proof} 
First we consider the case $\tau_2=0$. 
Fourier transform of the wave equation with respect to the space variable  yields
\begin{equation*}
\begin{aligned}
     -k^2\,\left(\mbox{Id} + \tau_1\frac{\partial}{\partial t}\right)\,\hat p 
     - \frac{1}{c_0^2}\,\frac{\partial^2 \hat p}{\partial t^2} 
     = - \frac{\hat\phi}{c_0^2}\,\delta(t)\,.
\end{aligned}
\end{equation*}
With Variation of Constants we obtain the characteristic equation
$$
         \lambda^2 + \tau_1\,c_0^2\,k^2\,\lambda + c_0^2\,k^2 = 0
$$
which has the solution
$$
          \lambda_{1,2} = - \mu(k) \pm\,\i\, \vartheta(k)\,
$$
with $\mu$ and $\vartheta$ as in~(\ref{defs}). The ansatz
$$
       C_1\,e^{\lambda_1\,t}\,H(t) + C_2\,e^{\lambda_2\,t}\,H(t)
$$
yields the following solution in the wave number - time domain 
$$
   \hat p(\k,t) 
       = \hat \phi(\k)\,\frac{e^{-\mu(k)\,t}}{(2\,\pi)^{3/2}}\,\frac{ \sin(\vartheta(k)\,t)}{\vartheta(k)}\,H(t)\,. 
$$ 
The solution for the case $\tau_2\not=0$ is obtained by application of the operator 
$\left(\mbox{Id} + \tau_2\,\frac{\partial }{\partial t}\right)$ to the last result, which proves the lemma. 
\end{proof}

\subsection{Waves backward in time}

Because the time reversed wave equation is obtained by the substitution $t\to -t$, the proof of Lemma~\ref{lemm:tva} 
implies the following lemma.

\begin{lemm}\label{lemm:tvb} 
Let~(\ref{defs}) hold. If $q$ solves  
\begin{equation*}
\begin{aligned}
     \left(\mbox{Id} - \tau_1\frac{\partial}{\partial t}\right)\,\Delta\,q
     - \frac{1}{c_0^2}\,\frac{\partial^2 q}{\partial t^2} 
     = - \frac{\phi}{c_0^2}\,\left(\mbox{Id} - \tau_2\,\frac{\partial }{\partial t}\right)\,
            \delta(t)
\end{aligned}
\end{equation*} 
and $\x\mapsto q(\x,t)$ is tempered for $t\in\R$, then 
$$
   \hat q(\k,t) 
       = - \hat\phi(\k)\,\left(\mbox{Id} - \tau_2\,\frac{\partial }{\partial t}\right)\left( 
          \frac{e^{\mu(k)\,t}}{(2\,\pi)^{3/2}}\,\frac{ \sin(\vartheta(k)\,t)}{\vartheta(k)}\,H(t) \right) 
$$ 
for $k:=|\k|$ with $\k\in\R^3$.
\end{lemm}

\begin{rema}\label{rema:subst}
From this lemma we see that time reversal is equivalent to the following substitutions:
$$
\mu\to -\mu,\,\qquad\tau_1\to-\tau_1\qquad \mbox{and} \qquad\hat \phi\to -\hat \phi \,
$$
in the wave vector-time domain. 
\end{rema}

In order to guarantee the solvability of the time reversed wave equation, either $\phi$ is required to satisfy a "source" 
condition or $\phi$ is regularized.  
For this purpose we introduce the following spaces $\mathcal{G}_D$
\begin{itemize}
\item for $D>0$
           \begin{equation}\label{invgs1}
               \phi \in \mathcal{G}_D \quad :\Leftrightarrow\quad  
                  \mbox{$\exists \varphi_* \in  L^2(\R^3)$ such that $\phi = \phi_* *_\x g_D$} \,,
          \end{equation}

\item for $D<0$
           \begin{equation}\label{invgs2}
               \phi \in \mathcal{G}_D \quad :\Leftrightarrow\quad  
                   \mbox{$\phi^* := \phi *_\x g_{-D}$ exists and $\phi^* \in  L^2(\R^3)$} \,,
\end{equation}

\end{itemize}
where $g_D$ denotes the Gaussian function
\begin{equation}\label{gauss}
   g_D(\x) := (4\,\pi\,D)^{-3/2}\,e^{-\frac{|\x|^2}{4\,D}}\, \qquad (D>0) \,.
\end{equation}
For $D>0$ it follows that $\mathcal{G}_D \subset L^2(\R^3)$ and for $D<0$ $\phi \in \mathcal{G}_D$ 
means formally that the Fourier transform of $\phi$ with respect to the space variable $\x$ has a factor that is exponentially 
increasing with respect to the norm of the wave vector $|\k|$. 
Moreover, we introduce the following \emph{regularization operator} for $D>0$ by
\begin{equation}\label{defRd} 
      \mathcal{R}_D\,\varphi := g_D *_\x \varphi     \qquad\mbox{for}\qquad \varphi\in L^2(\Omega)\,. 
\end{equation}
For $\varphi\in L^2(\Omega)$ we have $\mathcal{R}_D\,\varphi \in \mathcal{G}_D$ and $\varphi=(\mathcal{R}_D\,\varphi)_*$.

\begin{lemm}\label{lemm:tv2} 
Let~(\ref{defs}) hold, $\tau_2=\tau_1$, 
\begin{equation}\label{defs2}
\begin{aligned}
    \epsilon,\,T>0,\quad D:=2\,\left(\frac{c_0}{k_c}+\epsilon\right)\,T 
\end{aligned}
\end{equation}
and $\phi\in \mathcal{G}_D$. 
Then the solution $q$ of wave equation in Lemma~\ref{lemm:tvb} exists on the interval $[0,T]$ and  
$q(\cdot,t) \in L^2(\R^3)$ for $t\in [0,T]$. 
\end{lemm}

\begin{proof}   
Let $k:=|\k|$ for $\k\in\R^3$. We recall that $\hat g_D(\k)=(2\,\pi)^{-3/2}\,e^{-D\,|\k|^2}$. For $\phi\in \mathcal{G}_D$ 
and every $t\in [0,T]$ we have 
$$
    \hat q(\k,t) 
       = - \hat \phi_*(\k) \,e^{-D\,k^2}\,\frac{e^{\mu(k)\,t}}{(2\,\pi)^{3/2}}\,
          \left[  (1 - \tau_1\,\mu(k))\,\frac{ \sin(\vartheta(k)\,t)}{\vartheta(k)} 
                     - \tau_1\,\cos(\vartheta(k)\,t)  \right]
$$ 
according to Lemma~\ref{lemm:tvb}. The functions $k\mapsto \frac{ \sin(\vartheta(k)\,t)}{\vartheta(k)}$ 
and $k\mapsto \cos(\vartheta(k)\,t)$ are real valued and bounded on $D_c := \left\{\k\in\R^3\,\left|\, |\k|\leq k_c\right.\right\}$. 
On $\R^3\backslash D_c$ these functions are real valued but not bounded. 
Because of 
$$
   \i\,\vartheta_0(k)
       := \vartheta(k) 
        = \i\, c_0\,k\,\sqrt{\frac{k^2}{k_c^2} - 1} 
       \qquad\mbox{for}\qquad  k > k_c \,,
$$
$$
     \sin(\i\,x) = \frac{\i}{2}\,(e^x - e^{-x})  \qquad\mbox{and}\qquad 
     \cos(\i\,x) = \frac{1}{2}\,(e^x + e^{-x})  \,, 
$$
we have
$$
   \frac{ \sin(\vartheta\,t)}{\vartheta} 
        = \frac{e^{\vartheta_0\,t} - e^{-\vartheta_0\,t}}{2\,\vartheta_0}  \qquad\mbox{and}\qquad  
   \cos(\vartheta\,t) 
        =  \frac{1}{2}\,(e^{\vartheta_0\,t} + e^{-\vartheta_0\,t})  
$$
on $\R^3\backslash D_c$. 
Thus $\hat q$ has the following representation on $\R^3\backslash D_c$:  
\begin{equation}\label{helpq}
   \hat q(\cdot,t) 
       = - \hat \phi_*\,\frac{e^{(\mu + \vartheta_0)\,t-D\,k^2}}{(2\,\pi)^{3/2}}\,
          \left[  \frac{1 - \tau_1\,\mu}{2\,\vartheta_0}\,(1 - e^{-2\,\vartheta_0\,t})
                     - \frac{\tau_1}{2}\,(1 + e^{-2\,\vartheta_0\,t})  \right] \,.
\end{equation}
From this and 
$$
        e^{(\mu(k) + \vartheta_0(k))\,t - D\,k^2} = \mathcal{O} \left( e^{-\epsilon\,T\,k^2} \right) 
              \qquad\mbox{for}\qquad  k\to\infty\, ,
$$
we infer that $\hat q(\cdot,t) \in L^2(\R^3)$ for $t\in [0,T]$, as was to be shown. 
\end{proof}

It is surprising, that in contrast to non-dissipative media, time reversal in dissipative media 
requires regularized data.

\begin{rema}\label{rema:tv2} 
a) Let $\varphi\in L^2(\R^3)$. We note that analogously as in Lemma~\ref{lemm:tv2}, it follows that the solution $p$ 
in Lemma~\ref{lemm:tva} satisfies ($\mu\to -\mu$, $\tau_1\to-\tau_1$ and $\hat \varphi\to -\hat \varphi$ 
cf. Remark~\ref{rema:subst})
\begin{equation*}
   \hat p(\cdot,t) 
       = \hat \varphi\,\frac{e^{(-\mu + \vartheta_0)\,t}}{(2\,\pi)^{3/2}}\,
          \left[  \frac{1 - \tau_1\,\mu}{2\,\vartheta_0}\,(1 - e^{-2\,\vartheta_0\,t})
                     + \frac{\tau_1}{2}\,(1 + e^{-2\,\vartheta_0\,t})  \right] \,.
\end{equation*}
Here the first exponential term is bounded but not exponentially decreasing and thus, in general, $\phi_T:=p|_{t=T}$ 
is not an element of $\mathcal{G}_D$ for $D$ as in~(\ref{defs2}). Of course if $\hat \varphi$ oscillates appropriately, 
then $\phi_T\in\mathcal{G}_D$ is possible. 
As a consequence, in general, the time reversal exists only for regularized data.\\
b) For PAT the time reversed wave $q^{PAT}$ is given by $\frac{\partial q}{\partial t}$, where $q$ is as in Lemma~\ref{lemm:tv2}. 
From~(\ref{helpq}), it follows that
$$
    \frac{\partial q}{\partial t} \in L^2(\R^3) 
     \quad \Leftrightarrow \quad 
            \hat {\phi_T}_*\, k^2 \,e^{(\mu + \vartheta_0)\,t-D\,k^2}\,
                          \in L^2(\R^3)     
     \quad \Leftrightarrow \quad 
     \phi_T\in \mathcal{G}_D
$$ 
for $D:=2\,\left(\frac{c_0}{k_c}+\epsilon\right)\,T$ and therefore the time reversal problem for PAT with the 
regularized data $\mathcal{R}_D\,\phi_T$ has a solution if $\varphi\in L^2(\Omega)$. 
\end{rema}

\begin{rema}\label{rema:tv3} 
We note that a different and challenging regularization can be used that modifies $\phi_T$ only at wave vectors $\k$ 
with large norm such that the oscillations in $\hat \phi_T$ guarantee time reversal and 
$supp(\mathcal{R}(\phi_T))\subseteq \Omega$. But this issue is beyond the scope of this paper. 
\end{rema}

\section{The operator $\J_T:\mathcal{G}_D\to L^2(\R^3)$}
\label{sec-opJ}

Let $\mathcal{G}_D$ be defined as in~(\ref{invgs1}) and~(\ref{invgs2}). For $J\in\mathcal{G}_{-D}$ and $\phi\in \mathcal{G}_D$, 
we have formally 
$$
           \F\{J *_\x \phi\} = \hat J \,\hat\phi = J^*\,\hat g_D^{-1} *_\x \phi_*\,\hat g_D = J^* *_\x \phi_* \,. 
$$
Because $J^*,\,\phi_* \in L^2(\R^3)$, the last product as well as its inverse Fourier transform exist. 
Hence we define the convolution of $J\in\mathcal{G}_{-D}$ and $\phi\in \mathcal{G}_D$ by 
\begin{equation}\label{defconvJ}
    J *_\x \phi  := J^* *_\x \phi_* = \F^{-1}\{ \hat J^* \, \hat \phi_*\}\,.
\end{equation}
and consider $J$ as the convolution kernel of an operator $\J:\mathcal{G}_D\to L^2(\R^3)$.

\begin{prop}\label{prop:uT}
Let~(\ref{defs}) hold, $T$, $D$ be as in Lemma~\ref{lemm:tv2} and 
\begin{equation}\label{defJ}
    \hat J(\k)
        := (2\,\pi)^{-3/2}\,\frac{\sin^2(\vartheta(\k)\,T)}{\vartheta^2(\k)} 
        \qquad\mbox{for}\qquad  \k\in\R^3\,. 
\end{equation}
Then $\hat J:=\hat w(\cdot,T)$ holds, where $w$ is the solution of wave equation~(\ref{waveequ}) with $\xi(\x)=\delta(\x)$ 
and $J\in \mathcal{G}_{-D}$. If $\xi\in \mathcal{G}_{D}$, then the solution of~(\ref{waveequ}) satisfies 
$w(\cdot,t)=\J_t\,\xi\in L^2(\R^3)$ for $t\in [0,T]$. 
\end{prop}

\begin{proof}
The first claim follows from the fact that $\hat w$ solves
\begin{equation*}
     -|\k|^  2\,\hat w 
        + \frac{\tau_1^2\,c_0^2}{4}\,|\k|^4\,\hat w 
        - \frac{1}{c_1^2}\,\frac{\partial^2 \hat w}{\partial t^2} 
          = - \frac{2}{(2\,\pi)^{3/2}\,c_1^2}  \qquad\mbox{with} \qquad c_1:=2\,c_0\,,
\end{equation*}
which is equivalent to wave equation~(\ref{waveequ}) with $\xi(\x)=\delta(\x)$.  

For the second claim. 
Let $\epsilon>0$, $D := 2\,\left(\frac{c_0}{k_c}+\epsilon\right)\,T$ and 
$k:=|\k|$ for $\k\in\R^3$. Because 
\begin{itemize}
\item $\vartheta(k) =: \i\,\vartheta_0(k)$ is real for $k\leq k_c$, 

\item $\vartheta(k)$ is imaginary for $k> k_c$ and $\sin(\i\,x) = \i\,\sinh(x)$ and 

\item $a(k) :=\lim_{k\to\infty} \left(k\,\sqrt{\frac{k^2}{k_c^2}-1} - \frac{k^2}{k_c}\right) = - \frac{k_c}{2}$,

\end{itemize}
it follows that
$$
    b(k) := \lim_{k\to\infty} 2\,\vartheta_0\,T - D\,k^2 
       = 2\,c_0\,T\,a(k) - 2\,\epsilon\,T\,k^2 
       =  - c_0\,k_c\,T - 2\,\epsilon\,T\,k^2 
$$
and
$$
   \hat J(\k) =  \hat J^*(\k) \,  \,e^{D\,|\k|^2} 
         \qquad\mbox{with}\qquad
   \hat J^*(\k) := (2\,\pi)^{-3/2}\,\frac{\sin^2(\vartheta(k)\,T)}{\vartheta^2(k)}\,e^{-D\,|\k|^2} \,
$$
holds, where $\hat J^*$ is oscillating similarly as $\frac{\sin^2(\vartheta(k))}{\vartheta^2(k)}$ for $k<k_c$ and 
decreases exponentially like $e^{b(k)}$ for $k>k_c$. From this and $\epsilon>0$, 
it follows that
$$
  \|\hat J^* \|_{L^2}^2
    = 4\,\pi\,\int_0^\infty \left| \frac{\sin^2(\vartheta(k))}{\vartheta^2(k)} \right| 
              \,e^{-D\,|\k|^2} \,k^2\,\d k 
              < \infty \,, 
$$
which proves $J\in \mathcal{G}_{-D}$. The last claim follows from the fact that the convolution is well-defined 
by~(\ref{defconvJ}) and $\J_t\,\xi\in L^2(\R^3)$ for $t\in [0,T]$ and $\xi\in \mathcal{G}_{D}$. 
\end{proof}

\begin{prop}\label{prop:Kd}
Let~(\ref{defs}) hold and $T$, $D$ be as in Lemma~\ref{lemm:tv2}. Then the operator $\J_T:\mathcal{G}_D\to L^2(\Omega)$ 
defined by~(\ref{defconvJ}) is compact if $\mathcal{G}_D$ is endowed with the $L^2-$norm. 
\end{prop}

\begin{proof}
Because the operator $\phi_*\mapsto J^* *_\x \phi_*$ maps $L^2(\R^3)$ into $L^2(\R^3)$ and its kernel $J^*$ lies $L^2(\R^3)$, 
$\J_T$ is a Hilbert-Schmidt operator and thus is compact (cf.~\cite{Alt02}).  
\end{proof}

\section{Derivation of the operator equation}
\label{sec-dopeq}

In this section we derive the operator equation~(\ref{fredinteq}). Throughout this section 
\begin{itemize}
\item [(AD~1)] (\ref{mainass}),~(\ref{defs}),~(\ref{defs2}) hold, $\varphi\in L^2(\Omega)$ with $\Omega\subset\R^3$ open and 
               bounded, $J$ is defined as in Proposition~\ref{prop:uT} and 
              \begin{equation}\label{newdata}
                \phi_T := \phi_T(\varphi) := p^{tv}|_{t=T}  \,, 
              \end{equation} 
              where $p^{tv}$ denotes the solutions of~(\ref{waveeqp})  with $\alpha_*$ defined by~(\ref{Modeltv}). 
              Moreover, $F[\phi_T]$ is defined by~(\ref{defimagf}). 
\end{itemize}
In particular, this means that the regularized time reversed thermo-viscous wave equation has a solution in $L^2(\R^3)$ 
(cf. Remark~\ref{rema:tv2}), the operator $\J_T$ is well-defined and compact and  
$$
   p^{tv}|_{t=T}     
       \approx p^{ksb} \left( \cdot, T + T_1(\cdot) \right)  \,
$$
holds according to Proposition~\ref{prop:shift2}. Here $p^{ksb}$ denotes the causal solutions of~(\ref{waveeqp}) with 
$\alpha_*$ defined by~(\ref{Modeltvksb}). That is to say, the following results hold approximately if time shifted 
causal data are used instead of the non-causal one. 

For the convenience of the reader we recall that 
$F[\mathcal{R}_D\,\phi_T] := q|_{t=T}$,
where $q$ solves the \emph{regularized time reversed thermo-viscous wave equation} 
\begin{equation*}
     \left(\mbox{Id} - \tau_1\frac{\partial}{\partial t}\right)\,\Delta\,q 
     - \frac{1}{c_0^2}\,\frac{\partial^2 q}{\partial t^2} 
     = \frac{\mathcal{R}_D\,(\phi_T)}{c_0^2}\,\left(\mbox{Id} - \tau_2\,\frac{\partial }{\partial t}\right)\,
            \delta'(t) \,.
\end{equation*} 
Notice that the substitution $t\to-t$ leads to $\delta'(-t) = -\delta'(t)$, 
i.e. the right hand side in the time reversed wave equation must have a positive sign for PAT. 
It is best to start with the simplified case $\tau_2=0$.

\begin{lemm}\label{lemm:tv3}
If (AD~1) holds and $\tau_2=0$, then  
\begin{equation*}
   F[\mathcal{R}_D\,\phi_T] = (\mbox{Id}  + c_0^2\,\Delta\,\J_T)\,\mathcal{R}_D\,\varphi \,.
\end{equation*}
\end{lemm}

\begin{proof} 
According to Remark~\ref{rema:tv2}, $\mathcal{R}_D\,(\phi_T)\in \mathcal{G}_D$ with 
$D:=2\,\left(\frac{c_0}{k_c}+\epsilon\right)\,T$ implies that $q(\cdot,t)$ exists for $t\in [0,T]$ and $q(\cdot,T) \in L^2(\R^3)$. 
Let $\hat G_\pm$ denote the solutions of
\begin{equation}\label{eqGpm}
\begin{aligned}
     -|\k|^2\,\left(\mbox{Id} \pm \tau_1\frac{\partial}{\partial t}\right)\,\hat G_\pm
     - \frac{1}{c_0^2}\,\frac{\partial^2 \hat G_\pm}{\partial t^2} 
     = \mp \frac{1}{c_0^2}\,
         \frac{\delta(t)}{(2\,\pi)^{3/2}} \,.
\end{aligned}
\end{equation}
For convenience we set $(\hat\phi_T)_D:=\mathcal{R}_D\,\hat\phi_T$ and $\hat\varphi_D:=\mathcal{R}_D\,\hat\varphi$.   
From the convolution theorem 
$$
  \F\{f\,g\}(\k) = (2\,\pi)^{3/2}\,\hat f(\k)\,\hat g(\k) \qquad\quad \k\in\R^3 \,,
$$
we get 
\begin{equation}\label{repphiTq}
   (\hat\phi_T)_D = (2\,\pi)^{3/2}\,\hat G_+|_{t=T} \,\hat\varphi_D \quad\mbox{and}\quad 
   \hat q|_{t=T} = (2\,\pi)^{3/2}\,\hat G_-|_{t=T} \,(\hat\phi_T)_D    
\end{equation}
and consequently  
$$
   \hat q|_{t=T} = (2\,\pi)^3\,\left[\frac{\partial \hat G_-}{\partial t} \,
                                \frac{\partial \hat G_+}{\partial t}\right]_{t=T} \,\hat\varphi_D\,.
$$
From Lemma~\ref{lemm:tva} and Lemma~\ref{lemm:tvb}, it follows for $t>0$ that 
\begin{equation*}
\begin{aligned}
  \frac{\partial \hat G_-}{\partial t} \, \frac{\partial \hat G_+}{\partial t}
      &=  \frac{1}{(2\,\pi)^3}\,\frac{\partial }{\partial t}\left(e^{\mu\,t}\,\frac{\sin(\vartheta\,t)}{\vartheta} \right)\,
          \frac{\partial }{\partial t}\left(e^{-\mu\,t}\,\frac{\sin(\vartheta\,t)}{\vartheta} \right)\,\\
      &=  \frac{1}{(2\,\pi)^3}\,\left(1 + (\cos^2(\vartheta\,t)-1) - \mu^2\,\frac{\sin^2(\vartheta\,t)}{\vartheta^2}\right) \,. 
\end{aligned}
\end{equation*}
With $\cos^2(\vartheta\,t)-1 = -\sin^2(\vartheta\,t)$ and $\vartheta^2+\mu^2=c_0^2\,|\k|^2$, the last result simplifies to 
\begin{equation*}
\begin{aligned}
  \frac{\partial \hat G_-}{\partial t} \, \frac{\partial \hat G_+}{\partial t}
      =  \frac{1}{(2\,\pi)^3}\,\left(1 - c_0^2\,|\k|^2\,\frac{\sin^2(\vartheta\,t)}{\vartheta^2}\right)\,. 
\end{aligned}
\end{equation*}
With
\begin{equation}\label{relLaplace}
        - \F^{-1}\left\{\frac{|\k|^2}{(2\,\pi)^{3/2}}\right\}(\x) \,*_\x =  \Delta \,,
\end{equation}
we obtain finally 
$$
   F[\mathcal{R}_D\,\phi_T] 
       =  q|_{t=T} 
       = \left(\delta(\x) 
         + c_0^2\,\Delta\, \F^{-1}\left\{\frac{\sin^2(\vartheta\,t)}{(2\,\pi)^{3/2}\,\vartheta^2}\right\} \right) *_\x \varphi_D\,, 
$$
which concludes the proof. 
\end{proof}

We now come to the case $\tau_2=\tau_1\geq 0$.

\begin{theo}\label{theo:tv2}
If (AD~1) holds and $\tau_2=0$ or $\tau_2=\tau_1$, then 
\begin{equation*}
\begin{aligned}
    F[\mathcal{R}_D\,\phi_T] =  \left(\mbox{Id} + \tau_2^2\,c_0^2\,\Delta\right)^2\,\mathcal{R}_D\,\varphi 
                    + c_0^2\,\Delta\,\J_T \,\mathcal{R}_D\,\varphi \,
\end{aligned}
\end{equation*}
\end{theo}

\begin{proof} 
According to Remark~\ref{rema:tv2}, $\mathcal{R}_D\,(\phi_T)\in \mathcal{G}_D$ with 
$D:=2\,\left(\frac{c_0}{k_c}+\epsilon\right)\,T$ implies that $q(\cdot,t)$ exists for $t\in [0,T]$ and 
$q(\cdot,T) \in L^2(\R^3)$. 
As in the proof of Lemma~\ref{lemm:tv3}, it follows that 
\begin{equation}\label{repq}
   \hat q|_{t=T} = (2\,\pi)^3\, [A_- \,A_+]_{t=T} \,\hat\varphi_D \,, 
\end{equation}
where 
$$
   A_\pm
     := \left(\mbox{Id} \pm \tau_1\,\frac{\partial }{\partial t}\right)\,
                         \frac{\partial \hat G_\pm}{\partial t}\,.
$$
and $G_\pm$ are defined by~(\ref{eqGpm}). 
From the Lemmata~\ref{lemm:tva} and~\ref{lemm:tvb}, it follows for $t>0$ that 
\begin{equation*}
\begin{aligned}
    (2\,\pi)^{3/2}\,e^{\pm \mu\,t}\, A_\pm 
      = \pm [- \mu + \tau_1\,\mu^2 - \tau_1\,\vartheta^2]\,\frac{\sin(\vartheta\,t)}{\vartheta} 
        + [1-2\,\tau_1\,\mu]\,\cos(\vartheta\,t) \,, 
\end{aligned}
\end{equation*}
which implies with 
$$
      (a-b)\,(a+b)=a^2-b^2    \qquad\mbox{and}\qquad  
      \cos^2(\vartheta\,t) = 1 - \vartheta^2\,\frac{\sin^2(\vartheta\,t)}{\vartheta^2}
$$
that 
\begin{equation*}
\begin{aligned}
    (2\,\pi)^3\, A_-\,A_+
      = [1-2\,\tau_1\,\mu]^2 
        - \left\{[- \mu + \tau_1\,\mu^2 - \tau_1\,\vartheta^2]^2 + [1-2\,\tau_1\,\mu]^2\,\vartheta^2\right\}
          \,\frac{\sin^2(\vartheta\,t)}{\vartheta^2} \,.
\end{aligned}
\end{equation*}
Employing $\mu^2 + \vartheta^2 = c^2\,|\k|^2$ to the last result yields
\begin{equation*}
\begin{aligned}
    (2\,\pi)^3\, A_-\,A_+
      = [1-2\,\tau_1\,\mu]^2 
        - (\mu^2 + \vartheta^2)\,\frac{\sin^2(\vartheta\,t)}{\vartheta^2} \,.
\end{aligned}
\end{equation*}
Inserting this in~(\ref{repq}) results in
$$
   F[\mathcal{R}_D\,\phi_T] 
       = \hat q|_{t=T} 
       =  [1-2\,\tau_1\,\mu]^2 \,\hat\varphi_D
          - (\mu^2 + \vartheta^2)\,\frac{\sin^2(\vartheta\,T)}{\vartheta^2}\,\hat\varphi_D \,, 
$$
which is due to~(\ref{relLaplace}) and  
$$
      \F^{-1}\left\{\frac{\mu^2 + \vartheta^2}{(2\,\pi)^{3/2}}\right\} *_\x = - c_0^2\,\Delta
$$ 
is equivalent to
$$
    F[\mathcal{R}_D\,\phi_T] 
       =  [1 + \tau_1^2\,c_0^2\,\Delta]^2 \,\mathcal{R}_D\,\varphi
          + c_0^2\,\Delta \,\J_T\,\mathcal{R}_D\,\varphi\,.
$$
As was to be shown. 
\end{proof}

\section{Properties of the operator equation}
\label{sec-popeq}

Operator equation~(\ref{fredinteq}), which follows from Theorem~\ref{theo:tv2}, can be written as follows 
\begin{equation}\label{opeqA}
  \mathcal{A}(\mathcal{R}_D\,\varphi) = f\,,
\end{equation}
where $f := F_1 [\phi_T|_{\Omega},\beta]$ and the operator $\mathcal{A}:\mathcal{G}_D\to L^2(\R^3)$ is defined by 
\begin{equation}\label{defA}
    \mathcal{A} := (\mbox{Id} + \tau_1^2\,c_0^2\,\Delta)^2 + c_0^2\,\Delta\, \J_T  \,.
\end{equation} 
Here $\mathcal{G}_{D}$ and $\mathcal{G}_{-D}$ are defined as in~(\ref{invgs1}) and~(\ref{invgs2}), respectively.   
Because the kernel of the operator $\J_T$ defined by~(\ref{defJ}) is an element of $\mathcal{G}_{-D}$ for 
\begin{equation}\label{defd}
   D:=2\,\left(\frac{c_0}{k_c}+\epsilon\right)\,T  \qquad\mbox{with}\qquad\epsilon>0\,
\end{equation}
and $\J_T\,\xi$ is well-defined only if $\xi\in\mathcal{G}_D$ (cf. Section~\ref{sec-opJ}), it is required that 
$\mathcal{R}_D\,\varphi\in \mathcal{G}_D$. But this means nothing else than $\varphi \in  L^2(\Omega)$. Moreover, 
$\Delta^2\,\mathcal{R}_D\,\varphi$ exists and lies in $L^2(\R^3)$. Hence we get the following proposition. 

\begin{prop}
A necessary condition for the solvability of operator equation~(\ref{opeqA}) is that 
$$
     \varphi \in  L^2(\Omega) \qquad\mbox{with $D$ defined as in~(\ref{defd}).}
$$
\end{prop}

The next two propositions discuss the injectivity and surjectivity of the operator $\mathcal{A}$. 

\begin{prop}
If (AD~1) holds, then $\mathcal{A} \circ \mathcal{R}_D:L^2(\Omega)\to L^2(\R^3)$ is injective. 
\end{prop}

\begin{proof}
We show that the null space $\mathcal{N}(\mathcal{A} \circ \mathcal{R}_D)$ contains only the 
zero function. We note that $g_D$ defined as in~(\ref{gauss}) satisfies $\hat g_D(\x) := (2\,\pi)^{-3/2}\,e^{-D\,k^2}$ for  
$k=|\k|$ with $\k\in\R^3$  and that $k_c:= \frac{2}{\tau_1\,c_0}$.
Because the Fourier transform is an isometry on $L^2(\R^3)$, it follows that $\varphi\in \mathcal{N}(\mathcal{A} \circ \mathcal{R}_D)$ 
if and only if 
\begin{equation}\label{helpfred}
\begin{aligned}
   \varphi \in L^2(\Omega)  
              \quad\mbox{and}\quad 
   \left[\left(1 - 4\,\frac{k^2}{k_c^2}\right)^2 
                    - c_0^2\,k^2\, \hat J^2(\k) \right]\,e^{-4\,D\,|\k|^2}\,\hat\varphi = 0\, 
\end{aligned}
\end{equation}
with 
$$
    \hat J(\k)
        := \frac{\sin(\vartheta(\k)\,T)}{\vartheta(\k)} 
        \qquad\mbox{for}\qquad  \k\in\R^3\,. 
$$
We assume that $\varphi$ is not the zero function and prove a contradiction.  
Because $\varphi$ has compact support, the Paley-Wiener Theorem (cf.~\cite{Hoe03}) implies that the set $Z$ of zeros of 
$\hat \varphi$ cannot contain an open ball $B\subset\R^3$ and thus 
\begin{equation*}
   \left(1 - 4\,\frac{k^2}{k_c^2}\right)^2
                    - c_0^2\,k^2\, \frac{\sin^2(\vartheta(k)\,T)}{\vartheta(k)^2} = 0 \qquad\mbox{for}\qquad \k\in B\backslash Z\,
\end{equation*}
follows. But this is equivalently to
\begin{equation}\label{identity}
   \left(1 - 4\,\frac{k^2}{k_c^2}\right)^2\,\left(1 - \frac{k^2}{k_c^2}\right)
            = \sin^2(\vartheta(k)\,T)   \qquad\mbox{for}\qquad \k\in B\backslash Z\,.
\end{equation}
We recall that $\sin^2(\vartheta(k)\,T)$ (cf. Lemma~\ref{lemm:tv2}) oscillates between $\pm 1$ for 
$|k|\leq k_c$ and increases exponentially for $|k|> k_c$. Hence~(\ref{identity}) has finite many zeros on $[0,k_c]$ but 
no zeros on $(k_c,\infty)$, consequently identity~(\ref{identity}) is not true for $\k\in B\backslash Z$, 
where $B$ is an open ball in $\R^3$. Hence the assumption $\varphi\not=0$ is false, which proves the claim. 
\end{proof}

\begin{prop}
If (AD~1) holds, then $\mathcal{A}:\mathcal{G}_D\to L^2(\R^3)$ is not surjective. 
\end{prop}

\begin{proof}
Let $f\in L^2(\R^3)$ and $\varphi$ solve $\mathcal{A}(\varphi) = f$. Then from~(\ref{defA}), 
we infer for $k:=|\k|$ with $\k\in\R^3$ that 
$$
   \hat\varphi(\k) = \frac{\hat f(\k)}{\hat h(\k)} 
        \qquad\mbox{with}\qquad
   \hat h(\k) := (1 - \tau_1^2\,c_0^2\,k^2)^2 + c_0^2\,k^2\,\hat J^2(\k)\,,
$$
where $\hat h(\k)$ has zeros lying on a finite (but large) number of spheres. Hence $\hat\varphi(\k)\in L^2(\R^3)$ implies that a zero 
$\k_j$ of $\hat h$ of order $n_j$ is a zero of $\hat f$ of order $\geq n_j$. Because this is a real restriction on the 
space $L^2(\R^3)$, $\mathcal{A}$ cannot be surjective.  
\end{proof}

\section{Properties of the imaging functional}
\label{sec-propim}

Now we show under the assumption~(\ref{assT}) that 
\begin{equation}\label{relF1}
       F_1[\phi_T|_{\Omega},\beta]|_{\Omega} = \mathcal{R}_D\varphi|_{\Omega}  \qquad\mbox{for}\qquad 
       \tau_1=0\,
\end{equation}
and
\begin{equation}\label{relF2}
       \lim_{\tau_1\to 0} F_1[\phi_T|_{\Omega},\beta]|_{\Omega} = \mathcal{R}_D\varphi|_{\Omega}  \,
\end{equation}
hold for $D>0$, where $\mathcal{R}_D$ is defined as in~(\ref{defRd}). We recall that $\phi_T := p|_{t=T}$ and 
$\beta := p|_{\partial \Omega}$, where $p$ denotes the solution of~(\ref{waveeqp}) with complex attenuation law~(\ref{Modeltv}) 
and $\varphi\in L^2(\Omega)$.   Loosely speaking identity~(\ref{relF1}) means that $F_1[\phi_T|_{\Omega},\beta]_\Omega$ is 
similar to $\mathcal{R}_D \,\varphi$ if the attenuation is weak.

\begin{theo}\label{theo:relF1}
Let $\varphi\in L^2(\Omega)$. If $D>0$ and assumption~(\ref{assT}) holds, then identity~(\ref{relF1}) is true. 
\end{theo}

\begin{proof}
Let $\tau_1=0$ and $w(\cdot,t):=\J_t\,\varphi$ satisfy~(\ref{waveequ}). Then $w$ solves the standard wave equation
$$
     \Delta\,w 
        - \frac{1}{c_1^2}\,\frac{\partial^2 w}{\partial t^2} 
          = - \frac{2\,\varphi}{c_1^2}  \qquad\mbox{on $\R^3\times [0,T]$}\quad \mbox{with}\quad c_1:=2\,c_0\,.
$$
Consequently, $w$ can be written as 
\begin{equation*}
\begin{aligned}
      w(\x,T) = \frac{2}{c_1^2}\,\int_{\R^3}\int_0^\infty \frac{\delta\left(T-s-\frac{|\x-\y|}{c_1}\right)}
                                                                {4\,\pi\,|\x-\y|}\,\varphi(\y)\,\d s \,\d \y 
              = \frac{2}{c_1^2}\,\int_{\R^3}  \frac{\chi_{B_{c_1\,T}(\x)}(\y)}
                                                     {4\,\pi\,|\x-\y|}\,\varphi(\y)\,\d \y \,,
\end{aligned}
\end{equation*}
where we have used that 
$$
     \int_0^\infty \delta\left(T-s-\frac{|\x-\y|}{c_1}\right) \,\d s = 1 
              \qquad\mbox{if and only if}\qquad
     \y\in B_{c_1\,T}(\x)\,.
$$
Employing assumption~(\ref{assT}), $c_1=2\,c_0$ and $\Delta \frac{1}{|\x-\y|} = -4\,\pi\,\delta(\x-\y)$ to our last result yields
$$
  c_0^2\,\left[\Delta\,\J_T\,\varphi\right]_{\Omega} 
   = c_0^2\,\left[\Delta w(\x,T)\right]_{\Omega}  
     =  \frac{1}{8\,\pi}\,\Delta\,\int_{\supp(\varphi)}  \frac{\varphi(\y)}{|\x-\y|}\,\d \y
     =  -\frac{1}{2}\,\varphi\,, 
$$
which proves our claim. 
\end{proof}

\begin{theo}\label{theo:relF2}
Let $\varphi\in L^2 (\Omega)$. If $D>0$ and assumption~(\ref{assT}) holds, then identity~(\ref{relF2}) is true.
\end{theo}

\begin{proof}
Because $\mathcal{R}_D\,\varphi $ is bounded for $\varphi\in L^2(\Omega)$, it follows 
$$
    \lim_{\tau_1\to 0} a(\tau_1)\,\Delta\,\mathcal{R}_D\,\varphi 
      = \Delta\, \lim_{\tau_1\to 0} a(\tau_1)\,\mathcal{R}_D\,\varphi 
      = 0 \qquad \mbox{if}\qquad \lim a(\tau_1)\to 0\,.
$$
Hence we have 
$$
   \lim_{\tau_1\to 0} \left(\mbox{Id} + \tau_1^2\,c_0^2\,\Delta\right)^2\,\mathcal{R}_D\,\varphi 
    = \mathcal{R}_D\,\varphi \qquad\mbox{in}\qquad L^2(\R^3)\,. 
$$
It remains to show that 
$$
    \lim_{\tau_1\to 0} c_0^2\,\left[\Delta\,\J_T\,\mathcal{R}_D\,\varphi\right]_{\Omega} 
            = -\frac{1}{2}\,[\mathcal{R}_D\,\varphi]_{\Omega}  \qquad\mbox{in}\qquad L^2(\R^3)\,.
$$
Due to Theorem~\ref{theo:relF1} and its proof, this is equivalent to 
$$
   \lim_{\tau_1\to 0} c_0^2\,\left[\Delta w_{\tau_1} (\x,T)\right]_{\Omega} 
    =  c_0^2\,\left[\Delta w_0(\x,T)\right]_{\Omega}  
            \qquad\mbox{in}\qquad L^2(\R^3)\,, 
$$
where $w_{\tau_1}$ solves~(\ref{waveequ}) with $\hat \xi:=\hat g_D\,\hat\varphi$, i.e.  
(cf. Proposition~\ref{prop:uT})  
$$
    \hat w_{\tau_1}(\k,T) = \hat g_D(\k)\,\frac{\sin^2(\vartheta(|\k|)\,T)}{\vartheta^2(|\k|)}  \quad\mbox{with}\quad
    \vartheta(|\k|) = c_0\,|\k|\,\left(1-\frac{\tau_1^2\,c_0^2}{4}\,k^2\right) \,.
$$
Notice $\vartheta(|\k|) = c_0\,|\k|$ for $\tau_1=0$. Now we estimate $\|\hat w_{\tau}-\hat w_0\|_{L^2}^2$. 
Let $\epsilon>0$ be arbitrary. Because $\hat w_{\tau_1},\,\hat w_0\in L^2(\R^3)$, there exists $k_\epsilon>0$ such that 
$$
       \int_{\R^3\backslash V_\epsilon} \hat g_D^2(\k)\,|\hat w_{\tau_1}(\k,T)|^2\,\d \k \leq \frac{\epsilon}{8} 
       \qquad\mbox{for}\qquad \tau_1\in (0,\tau_0)\,,
$$
where $\tau_0$ is sufficiently small and $V_\epsilon:=B_{k_\epsilon}(\mathbf{0})$. Because 
$\hat w_{\tau_0}(\k,T)$ converges to $\hat w_0(\k,T)$ pointwise for $\tau_1\to 0$ and $\bar V_\epsilon$ is compact, it follows that 
$$
            |\hat w_{\tau_1}(\cdot,T)-\hat w_0(\cdot,T)|^2 \leq \frac{\epsilon}{2\,|V_\epsilon|} 
               \qquad\mbox{on}\quad \bar V_\epsilon \,\quad \mbox{for sufficiently small $\tau_1$.}
$$ 
With these two inequalities and $|a-b|^2\leq 4\,\max(|a|^2,|b|^2)$ for $a\,b\in\C$, we get
\begin{equation*}
\begin{aligned}
   \|\hat w_{\tau}-\hat w_0\|_{L^2}^2 
      \leq \frac{\epsilon}{2\,|V_\epsilon|} \,\int_{V_\epsilon}  \d \k 
               + 4\,\frac{\epsilon}{8} 
       = \epsilon\,\qquad \mbox{for sufficiently small $\tau_1$.}
\end{aligned}
\end{equation*}
Therefore $w_{\tau}$ converges to $w_0$ in $L^2(\R^3)$. Because $\Delta g_D$ is a linear combination of Gaussians, it follows 
that $\Delta w_{\tau}$ converges to $\Delta w_0$ in $L^2(\R^3)$, as was to be shown. 
\end{proof}

\begin{figure}[!ht]
\begin{center}
\includegraphics[height=4.6cm,angle=0]{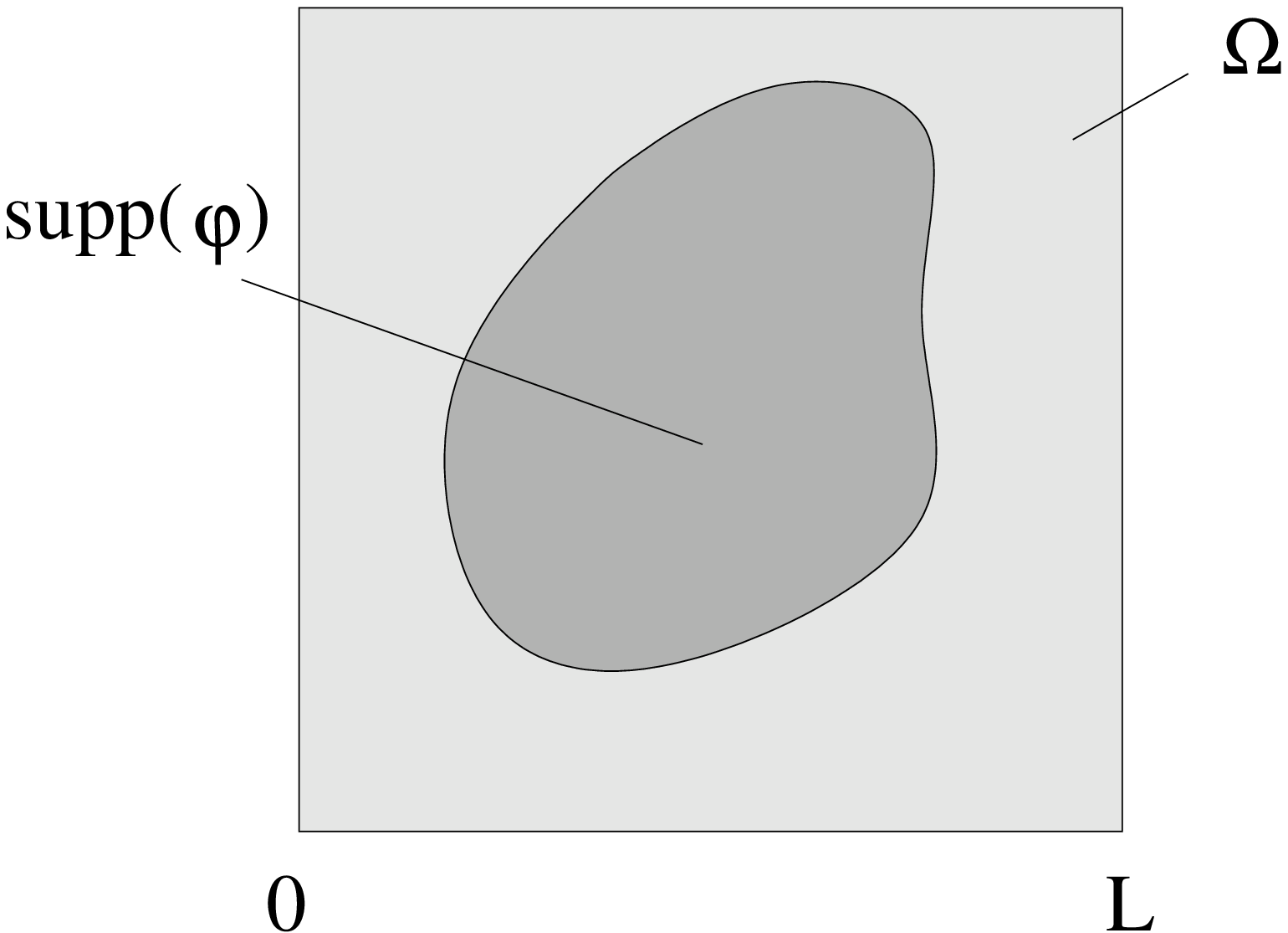}
\includegraphics[height=5.1cm,angle=0]{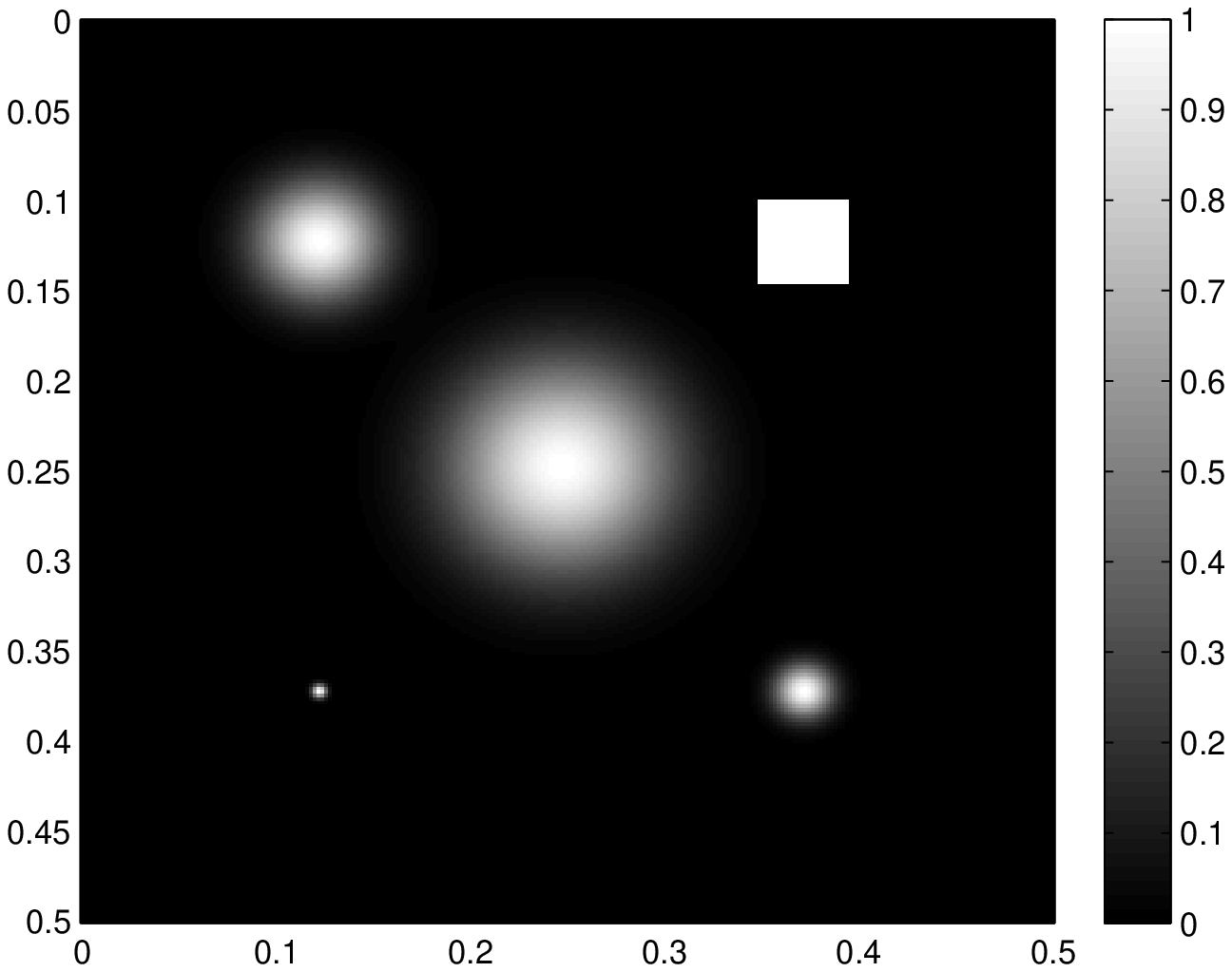} \\
\includegraphics[height=5.1cm,angle=0]{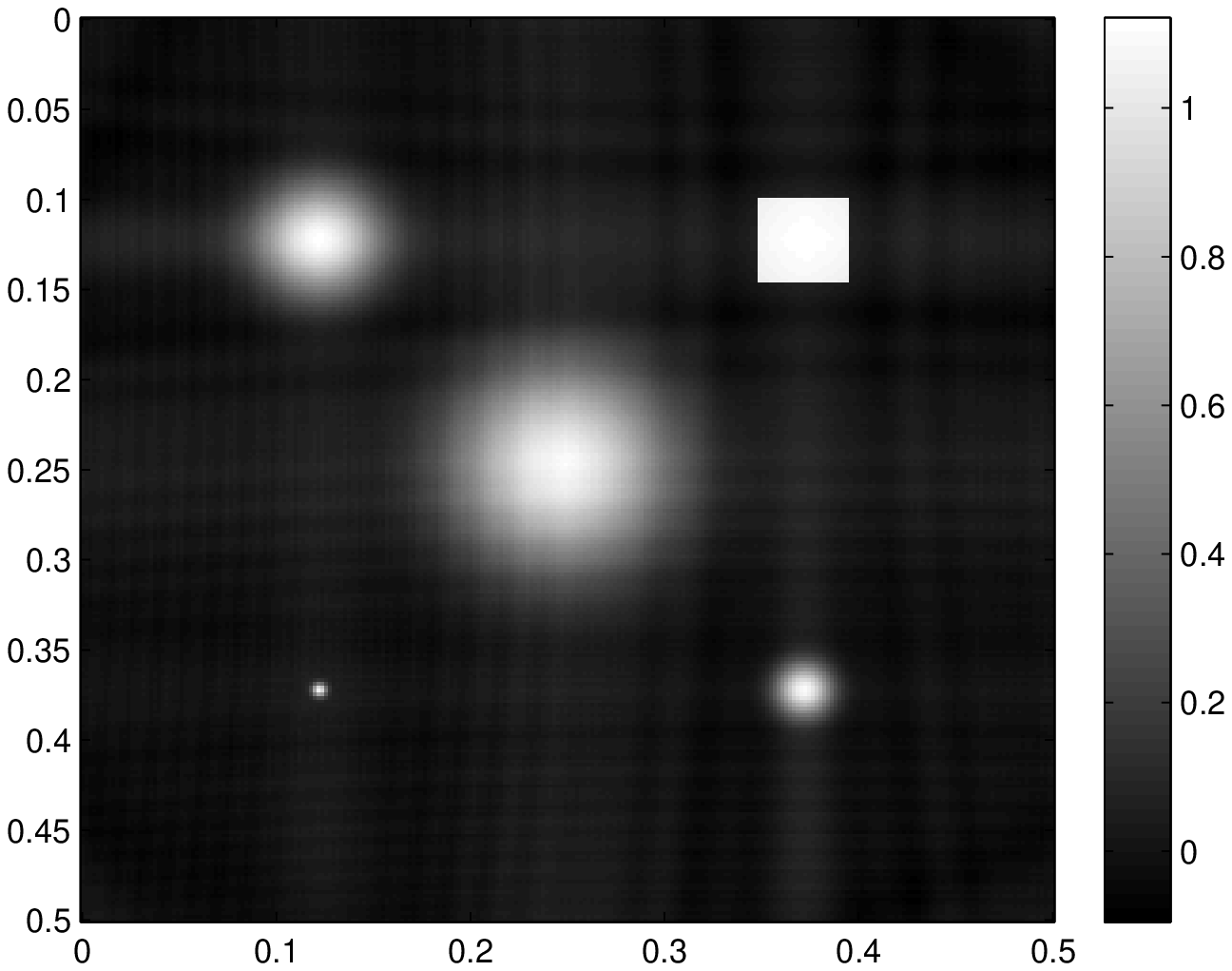}
\includegraphics[height=5.1cm,angle=0]{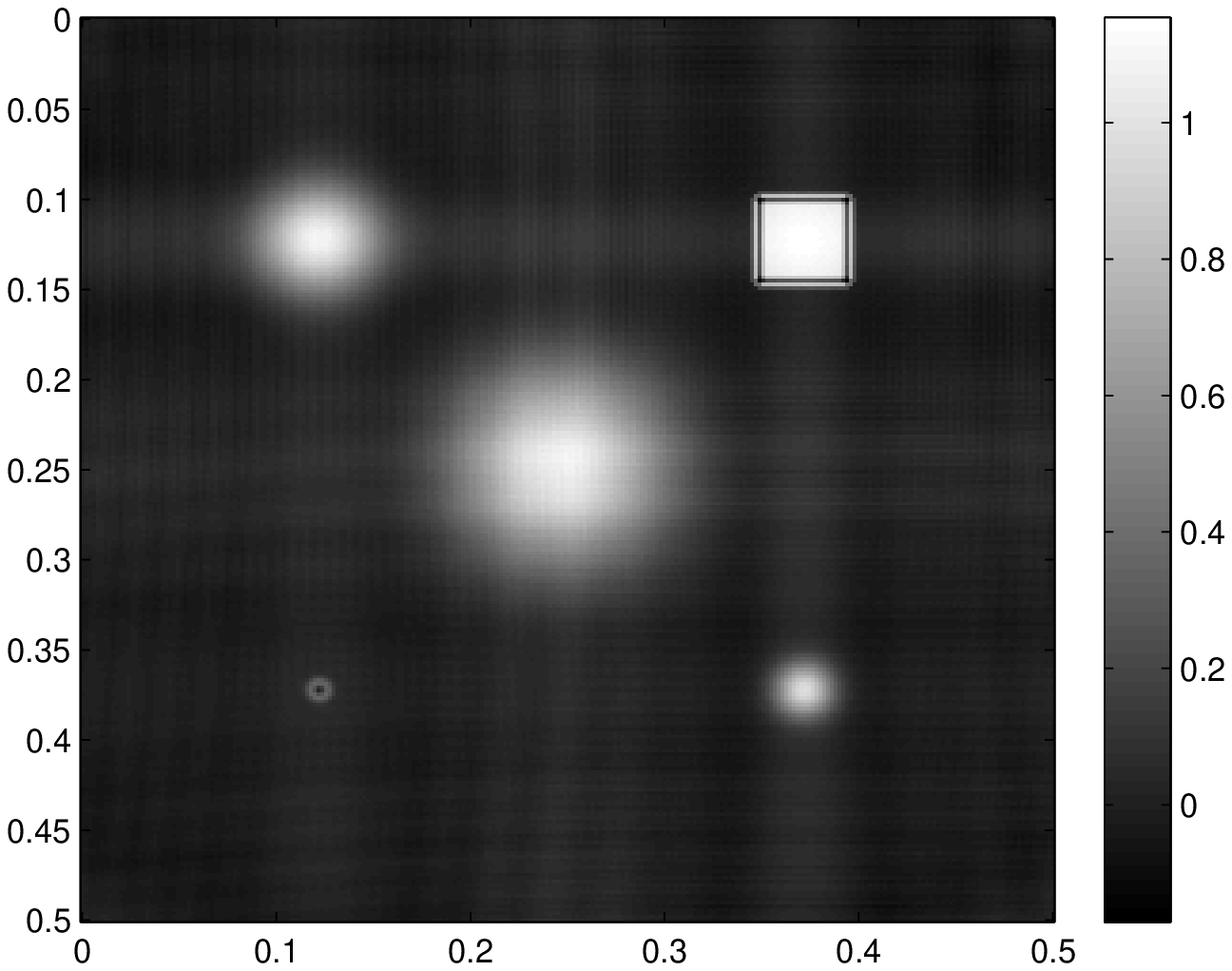}
\end{center}
\caption{The first row shows a visualization of the set-up with $L=0.5\,m$ and an example of the initial pressure function $\varphi$. 
The second row shows the time reversal image $\mathcal{I}$ for $\tau_1=10^{-9}\,s$ (water) and $\tau_1= 8.5\cdot 10^{-7}\,s$ 
(fictive), respectively. Small features are not well mapped by the time reversal image in case of strong dissipation.}
\label{fig:setup}
\end{figure}

\section{Simulation of the imaging functional}
\label{sec-sim}

In Section~\ref{sec-propim} we showed that 
\begin{equation}\label{relF2}
       \lim_{\tau_1\to 0} F_1[\phi_T|_{\Omega},\beta]|_{\Omega} = [\mathcal{R}_D\varphi]_{\Omega}  \,
\end{equation}
holds under an appropriate assumption, which means that the time reversal image (in the noise-free case) gives a good estimation 
of the initial pressure function if $\tau_1$ is sufficiently small. 
The goal of this section is to demonstrate that the parameter $\tau_1\approx 10^{-9}\,s$ for tissue similar to water is 
sufficiently small to ensure strong similarity between the initial pressure function $\varphi$ and the time reversal image
\begin{equation}\label{imageI}
\begin{aligned}
   \mathcal{I} := 2\,\left(\mbox{Id} + \tau_1^2\,c_0^2\,\Delta\right)^2\,\varphi + \Delta\,\mathcal{I}_0 
 \qquad\mbox{with}\qquad
   \mathcal{I}_0 := 2\,c_0^2\,\J_T\,\varphi \,.
\end{aligned}
\end{equation}
In order to save computation time we focus on the two dimensional case for which the results derived in this paper hold, too.

\begin{figure}[!ht]
\begin{center}
\includegraphics[height=5.1cm,angle=0]{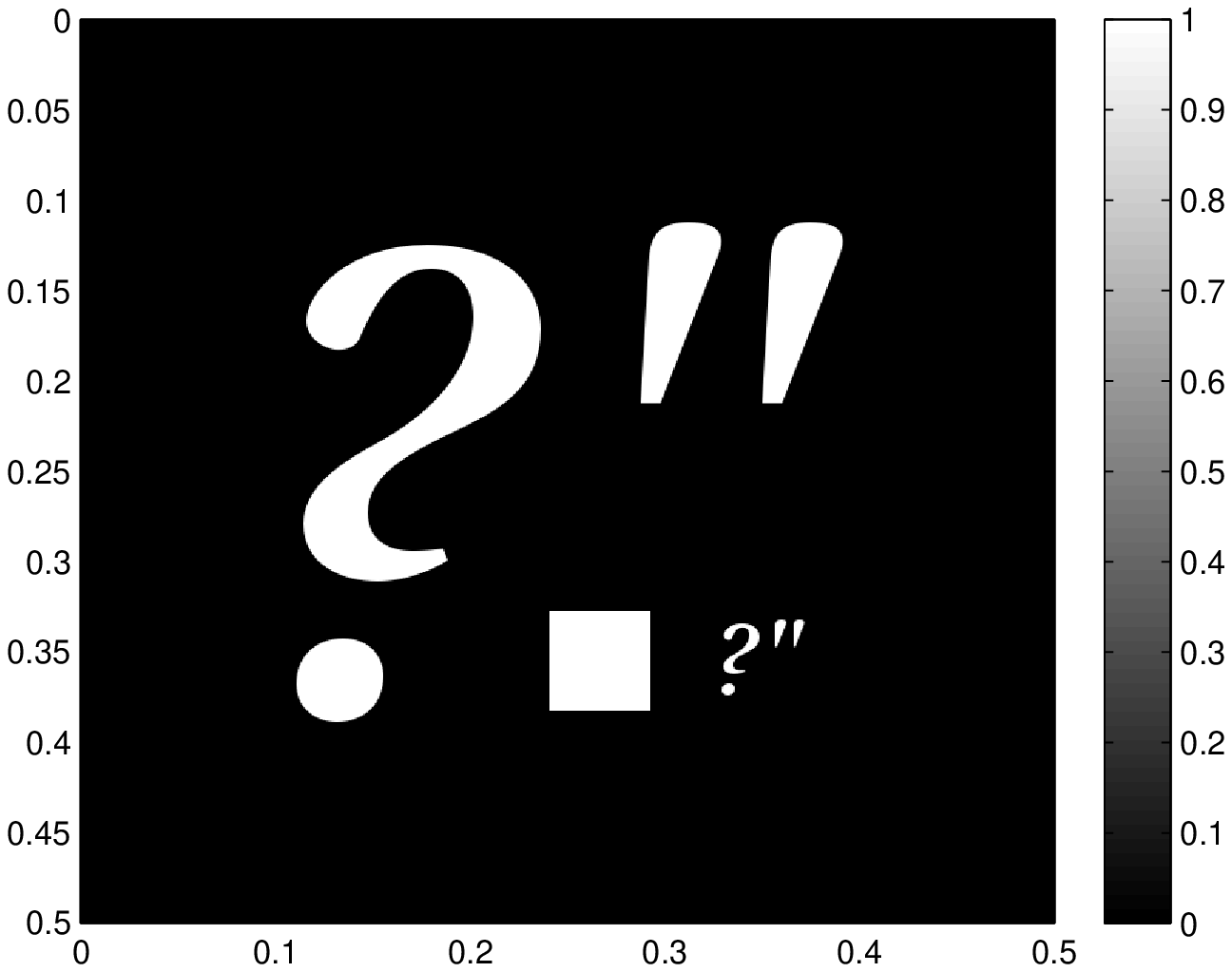} 
\includegraphics[height=5.1cm,angle=0]{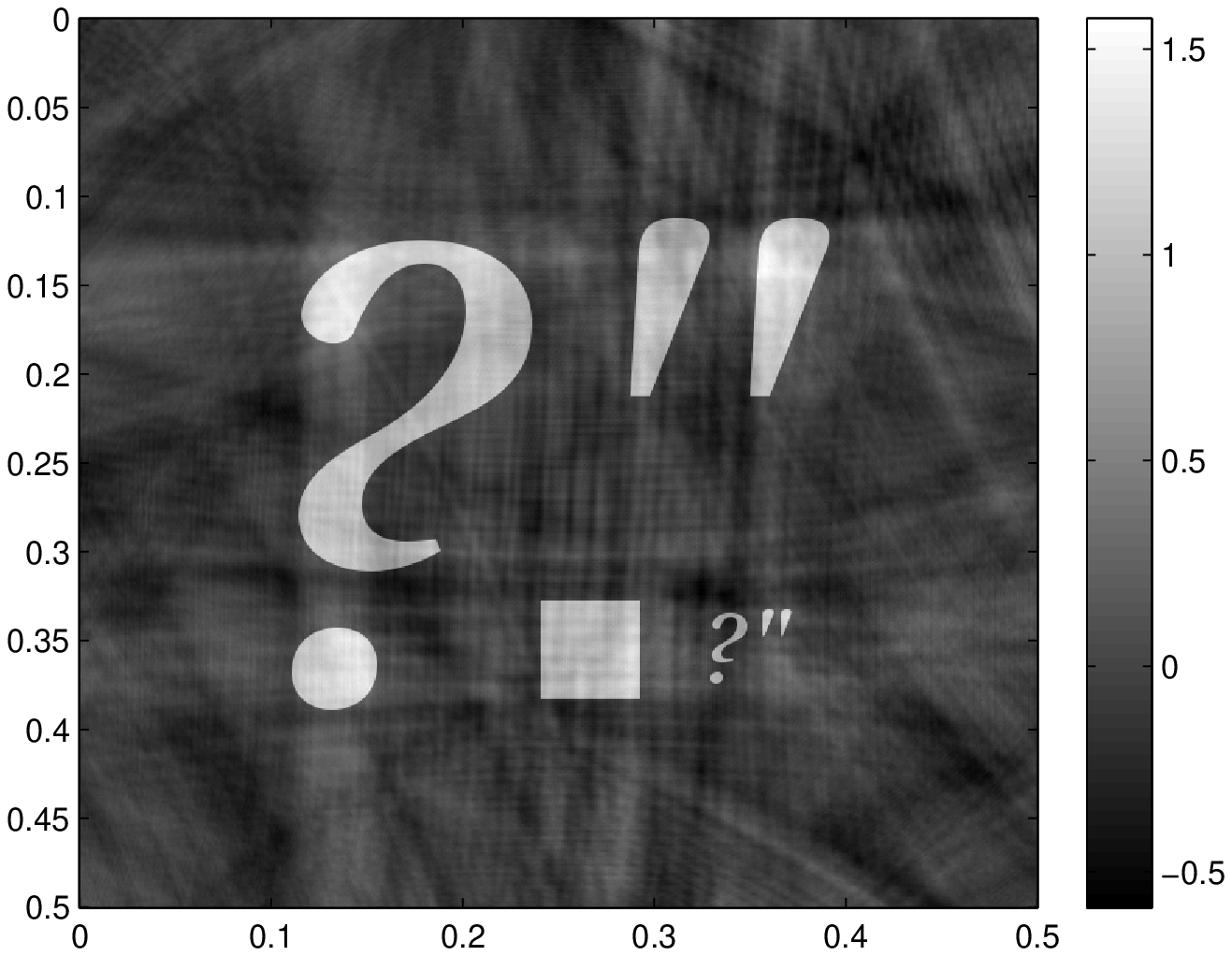} \\
\includegraphics[height=5.1cm,angle=0]{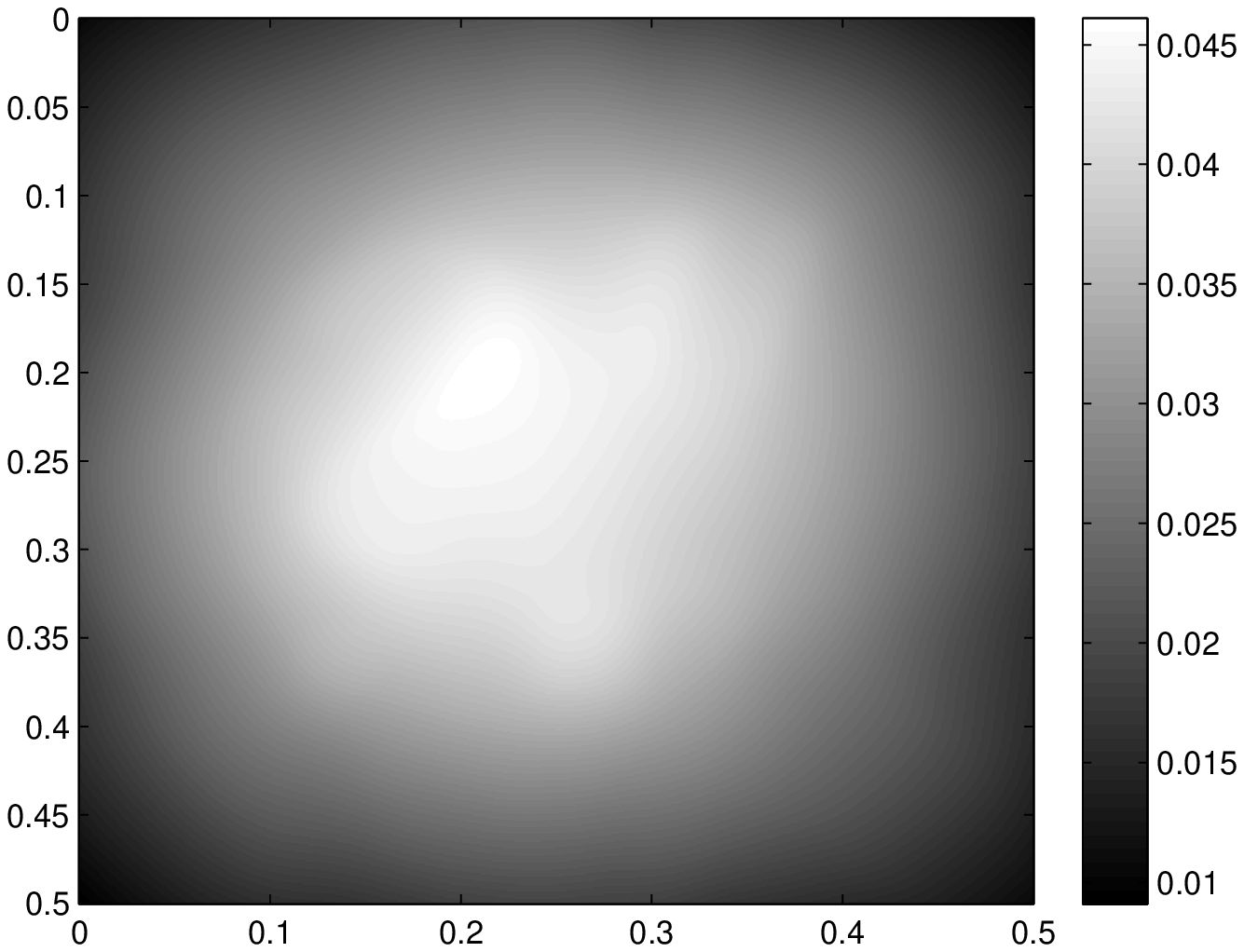} 
\includegraphics[height=5.1cm,angle=0]{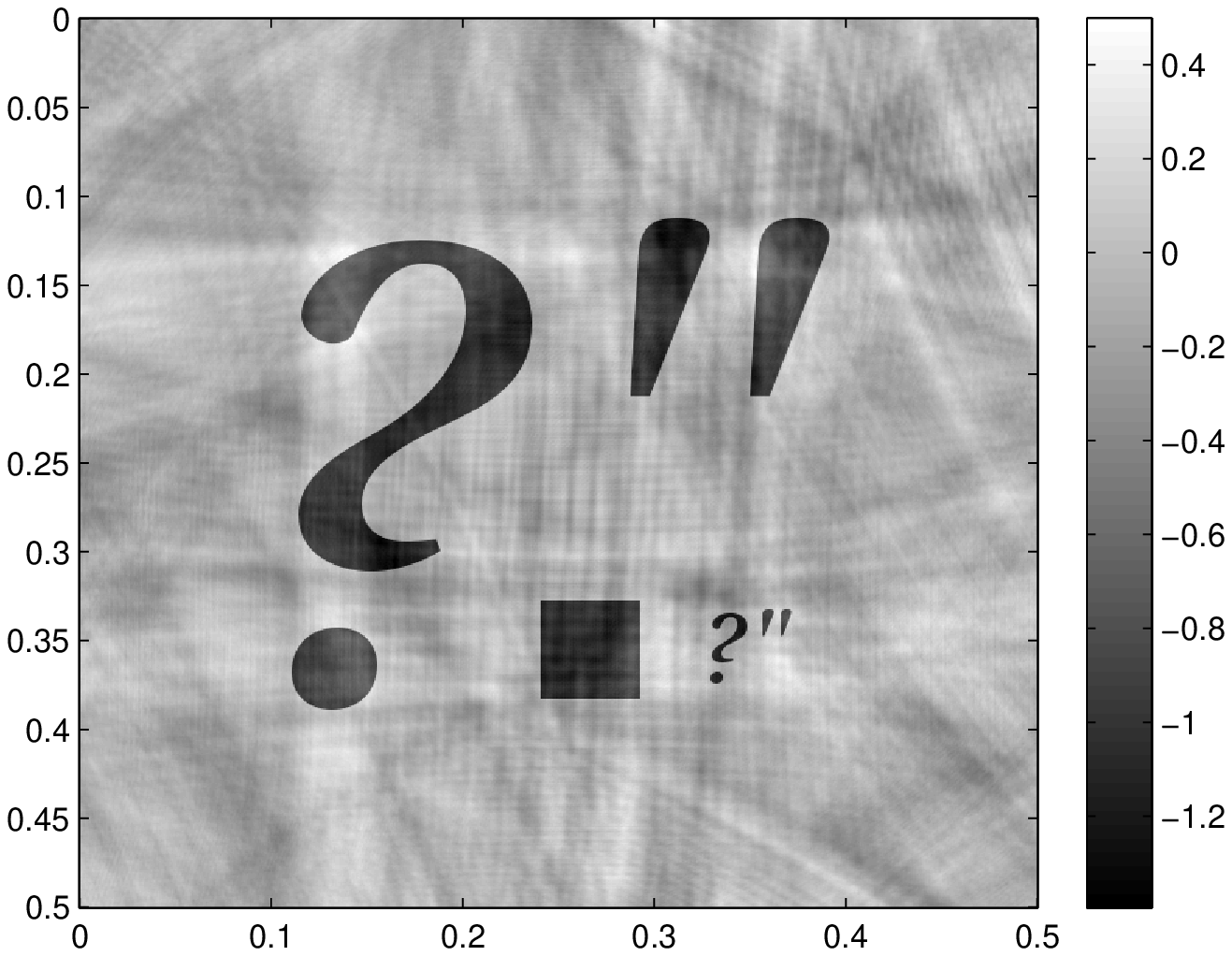}
\end{center}
\caption{The first row shows the initial pressure functions $\varphi$ and the respective time reversal image $\mathcal{I}$ 
for $\tau_1=10^{-9}\,s$.  
The function $\mathcal{I}_0$ which is up to a constant the solution of wave equation~(\ref{waveeqw}) and its Laplacian 
$\Delta\,\mathcal{I}_0$ are visualized in the last row.}
\label{fig:sim}
\end{figure}

The chosen set-up is visualized in Fig.~\ref{fig:setup}. 
For our simulations we have dropped the regularization operator $\mathcal{R}_D$ and $\J_t\,\varphi=w(\cdot,t)$ was calculated by 
solving the wave equation 
\begin{equation}\label{waveeqw}
     \Delta\,w 
        + \frac{\tau_1^2\,c_0^2}{4}\,\Delta^4\,w 
        - \frac{1}{c_1^2}\,\frac{\partial^2 w}{\partial t^2} 
          = - \frac{2\,\varphi}{c_1^2}  \qquad\mbox{on $\R^3\times [0,T]$} \qquad (c_1:=2\,c_0)\,.
\end{equation}
via the forward Euler method on $[0,4\,L]^2\times [0,T]$ with $L=0.5\,m$, $T=\frac{4\,L}{c_0}$, 
$$
    \Delta t = \frac{1}{2}\,\frac{\Delta x}{c_1}  \qquad\mbox{and}\qquad 
    \Delta x = \Delta y =\frac{L_0}{1020} \,.
$$
We note that the CFL-condition is satisfies for $\tau_1=0$.  

\begin{exam}
The first example is visualized in Fig.~\ref{fig:setup}. The circular peaks in $\varphi$ are  $C^\infty-$functions of the form 
$$ 
         \x \mapsto f(a-|\x-\mathbf{b}|^2) \qquad\mbox{with}\qquad  f(s) := e^{-1/s},\,\,\mathbf{b}\in\Omega  \qquad\mbox{and}\qquad
$$
$a\in \{3.6,\,14.4,\,32.4,\,57.6\}\,mm$. 
The right picture in the first row in Fig.~\ref{fig:setup} shows the initial pressure function $\varphi$ and the 
second row shows the time reversal images for $\tau_1=10^{-9}\,s$ (water) and $\tau_1= 8.5\cdot 10^{-7}\,s$ 
(fictive). We see that the smallest features is not well mapped for strong dissipation, it is twice as thick and its maximum 
intensity is about 30 percent of the correct maximum intensity. 
Numerical simulations show that the time reversal image blows up for $\tau_1= 10^{-6}\,s$, even if the time step size 
is decreased by the factor $1/8$. 
\end{exam}

\begin{exam}
A second numerical example for a piecewise constant function $\varphi$ is presented in Fig.~\ref{fig:sim}. The last row 
shows $\mathcal{I}_0$ which is up to a constant the solution of wave equation~(\ref{waveeqw}) and its Laplacian 
$\Delta\,\mathcal{I}_0$. We note that non-smooth initial pressure functions and stronger dissipation causes more artifacts 
in the solution of $\Delta\,\mathcal{I}_0$. As in the first example the time reversal image blows up for $\tau_1= 10^{-6}\,s$. 
\end{exam}

\section{Conclusions}

In this paper we showed that a method for solving PAT based on the non-causal thermo-viscous wave 
equation can be used if 
\begin{itemize}
\item the non-causal data are replaced by appropriately time shifted causal data and 
      
\item the size of $\supp(\varphi)$ is much smaller than its distance to the detectors.
\end{itemize}
In other words, performing an appropriate time shift of causal data permits the use of the non-causal 
thermo-viscous wave equation. 
Moreover, we showed that 
\begin{itemize}
\item strictly speaking the time reversal image exists only if the data are regularized, e.g. by the operator $\mathcal{R}_D$  
      for restriction $D>\tau_1\,c_0^2\,T$ (cf.~(\ref{defRd})) and that 
      
\item the regularized time reversal image for the case of dissipative media like water is very similar 
to a smoothed version of the initial pressure function $\varphi$. 
\end{itemize} 
If required, the time reversal image can be improved by solving operator equation~(\ref{opeqA}) with 
the time reversal image as right hand side. 

Above all, we would like to emphasize that this paper has analyzed the quality of an idealized estimation 
(noise-free case), but did not discussed the quality of a reconstruction using real noisy data.

\end{document}